\documentclass{article}

\usepackage{doi}
\usepackage{url}

\usepackage[utf8]{inputenc}
\usepackage[T1]{fontenc}

\usepackage{amsmath,amsfonts,amssymb,amsthm}
\usepackage{mathrsfs}
\usepackage{stmaryrd} 
\SetSymbolFont{stmry}{bold}{U}{stmry}{m}{n}

\usepackage{bm}

\usepackage{xspace}

\usepackage{graphicx}
\usepackage{svg}
\usepackage{subcaption}
\usepackage{enumitem}

\definecolor{newGreen}{HTML}{4DAF4A}

\usepackage{hyperref}
\usepackage{aliascnt}
\hypersetup{%
  colorlinks=true,
  linkcolor=blue,
  citecolor=blue,
  urlcolor=newGreen,
  linkbordercolor=red,
  pdfborderstyle={/S/U/W 1}
}
\usepackage{cleveref}

\newcommand{\ds}{\displaystyle}
\renewcommand{\O}{\Omega}

\newcommand{\bb}[1]{\left\llbracket#1\right\rrbracket}

\newcommand{\abs}[1]{\left|#1\right|}

\newcommand{\norm}[1]{\left\lVert#1\right\rVert}
\newcommand{\p}{\partial}

\newcommand{\vertiii}[1]{{\left\vert\kern-0.25ex\left\vert\kern-0.25ex\left\vert #1 
    \right\vert\kern-0.25ex\right\vert\kern-0.25ex\right\vert}}

\newcommand{\R}{\mathcal{R}}
\newcommand{\T}{\mathcal{T}}
\newcommand{\N}{\mathcal{N}}

\newcommand{\B}{\mathcal{B}}

\newcommand{\V}{\mathcal{V}}

\newcommand{\ttau}{\bm{\tau}}

\renewcommand{\c}{\mathbf{c}}
\newcommand{\x}{\mathbf{x}}

\newcommand{\g}{\mathbf{g}}
\newcommand{\n}{\mathbf{n}}

\newcommand{\hp}{\hat{p}}
\newcommand{\hq}{\hat{q}}

\newcommand{\hV}{\hat{\V}}
\newcommand{\hphi}{\hat{\phi}}
\newcommand{\hx}{\hat{\x}}

\newcommand{\hu}{\hat{u}}

\newcommand{\hGamma}{\hat{\Gamma}}

\newcommand{\hbeta}{\hat{\beta}}
\newcommand{\hK}{\hat{K}}
\renewcommand{\hbeta}{\hat{\beta}}

\newcommand{\cw}{\check{w}}

\newcommand{\tC}{\tilde{C}}

\newcommand{\PG}{P_{\Gamma}}
\newcommand{\RG}{R_{\Gamma}}
\newcommand{\ximid}{\xi^{\mathtt{mid}}}
\newcommand{\czeta}{\check{\zeta}}
\newcommand{\Nq}{\mathtt{N}^{(1)}_\mathtt{q}} 
 
\newcommand{\NTq}{\mathtt{N}_\mathtt{q}} 
\newcommand{\A}{\mathbf{A}}
\newcommand{\hlambda}{\hat{\lambda}}
\newcommand{\hB}{\hat{\B}}

\renewcommand{\L}{\mathscr{L}}

\newcommand{\E}{\mathcal{E}}
\DeclareMathOperator*{\argminA}{arg\,min}
\newcommand{\commentout}[1]{{}} 

\newcommand{\MATLAB}{\textsc{Matlab}\xspace}

\theoremstyle{definition}

\newtheorem{lemma}{Lemma}

\usepackage[top=2cm, bottom=2cm,left=2cm,right=2cm]{geometry}

\usepackage{lmodern}
\usepackage{microtype}

\usepackage[]{matlab-prettifier}
\lstnewenvironment{matlab}[1][] 
{\lstset{style=Matlab-bw,basicstyle=\mlttfamily,#1}}
{}
\lstnewenvironment{matlabblock}[1][] 
{\lstset{frame=single,style=Matlab-editor,basicstyle=\mlttfamily,escapechar=`,#1}}
{}

\makeatletter
\def\verbatim@font{\mlttfamily}
\makeatother

\usepackage[section]{easy-todo}

\title{{Construction of Basis Functions for \\
the Geometry Conforming Immersed Finite Element Method}}
\author{Slimane Adjerid\thanks{Department of Mathematics,  Virginia
Tech, Blacksburg, VA 24061 USA (adjerids@vt.edu, tlin@vt.edu)}
\and Tao Lin\footnotemark[1]\and
Haroun Meghaichi\thanks{Department of Mathematics, The Ohio State
University Columbus, OH 43210 USA (meghaichi.1@osu.edu)} }

\begin{document}
\maketitle

\begin{abstract}
The Frenet apparatus is a new framework for constructing high
order geometry-conforming immersed finite element functions for
interface problems. In this report, we present  a procedure for  constructing
the local IFE bases in some detail as well as  a new approach for constructing
orthonormal bases using the singular value decomposition of the local
generalized
Vandermonde matrix. A sample implementation in \MATLAB is provided to showcase
the simplicity and extensionability of the framework.
\end{abstract}


\section{Introduction}\label{sec:intro}

In this report, we discuss the implementation and the computational aspects
of the geometry-conforming immersed finite element (GC-IFE) functions
that were introduced in \cite{adjeridHighOrderGeometry2024} for
solving the following elliptic interface problem
\begin{subequations}\label{eqns:interface_problem}
\begin{equation}
\begin{cases}
-\nabla\cdot (\beta\nabla u )=f,& \text{on } \Omega^{-}\cup\Omega^+, \\
u|_{\p \Omega}=0.
\end{cases}\label{eqn:PDE}
\end{equation}
Here, the computational domain $\Omega \subset \mathbb{R}^2$ is
split by the interface $\Gamma$ into two subdomains $\Omega^-$ and
$\Omega^+$ and the diffusion coefficient $\beta$ is a piecewise
constant function  with $\beta|_{\Omega^{\pm}}=\beta^{\pm}>0$.
Furthermore, it is assumed that the solution satisfies the
homogeneous interface conditions
\begin{equation}
\bb{u}_{\Gamma}=0,\qquad \bb{\beta \nabla u\cdot \n}_{\Gamma}=0,
\label{eqn:interface_conds}
\end{equation}
where $\bb{v}_{\Gamma}$ denotes the \textit{jump} of $v$ across the
interface $\Gamma$, i.e.,
$\bb{v}_{\Gamma}=v^{+}|_{\Gamma}-v^{-}|_{\Gamma}$ where
$v^{\pm}=v|_{\O^{\pm}}$. When higher degree polynomials are
employed in the IFE functions, we also assume the so-called
extended jump condition:
\begin{equation}
\bb{\beta \p_{\n^j}\Delta u}_{\Gamma}=0,\quad j=0,1,\dots,m-2.
\label{eqn:interface_conds_ext}
\end{equation}
\end{subequations}

In our discussion, the solution domain $\Omega$ is assumed to be
a finite union of rectangles, and is partitioned into
an interface-independent uniform mesh $\T_h$ of rectangular elements
with {diameter} $h$. We call an element $K\in \T_h$ an interface
element if $\overset{\circ}{K}\cap \Gamma \neq \emptyset$, where
$\overset{\circ}{K}$ is the interior of $K$; otherwise, we call $K$ a
non-interface element. We use $\T_h^i$ and $\T_h^n$ to denote the set
of interface elements, and the set of non-interface elements,
respectively. In addition, we denote the set of edges, interior edges
and boundary edges of $\T_h$ by $\E_h,~\E_h^\circ$ and $\E_h^b$,
respectively. An IFE method will use standard finite element
functions on all non-interface elements, and it will use IFE functions
constructed according to the jump conditions
over all interface elements.

The GC-IFE functions are a new class of immersed finite element (IFE)
functions characterized by two distinct features:
\begin{enumerate} [label=(\arabic*)] 
\item
They can be constructed using polynomials of any desired degree;

\item
They exactly satisfy the jump conditions within each interface element.
\end{enumerate}
A key idea for constructing GC-IFE functions is the use of a
so-called Frenet transformation that relates the physical coordinates of a point
$X = (x, y)\in\O$ in a neighborhood of the interface curve $\Gamma$,
and the local coordinates $(\eta, \xi)$, which are defined  using the
Frenet apparatus. To be specific and without loss of
generality, let the interface curve $\Gamma$ be a regular curve
parametrized by $\g=(g_1(\xi),g_2(\xi)): [\xi_s,\xi_e]\to
\mathbb{R}^2$. The Frenet apparatus of $\Gamma$ consists of the
tangent vector $\ttau(\xi)$, the normal vector $\n(\xi)$, and the curvature
$\kappa(\xi)$, by which we introduce the Frenet transformation:
$P_\Gamma: \mathbb{R}^2 \rightarrow \mathbb{R}^2$ defined as follows
\begin{align}
\begin{bmatrix}
\eta \\
\xi
\end{bmatrix} \stackrel{P_\Gamma}{\longrightarrow}
\begin{bmatrix}
x(\eta,\xi) \\
y(\eta,\xi)
\end{bmatrix} = \x(\eta,\xi) =\PG(\eta,\xi)= \g(\xi)+\eta \n(\xi).
\label{eq:P_map}
\end{align}
It is well-known \cite{abateCurvesSurfaces2012} that $\Gamma$ has a
tubular neighborhood with a half-band width $\epsilon > 0$ denoted by
\begin{eqnarray}
N_\Gamma(\epsilon) = P_\Gamma([-\epsilon, \epsilon] \times [\xi_s,
\xi_e]) \label{eq:tubular}
\end{eqnarray}
in which the transformation $P_\Gamma(\eta, \xi): [-\epsilon,
\epsilon]\times [\xi_s, \xi_e] \rightarrow N_\Gamma(\epsilon)$ is
one-to-one and onto.
Therefore, $P_\Gamma(\eta, \xi)$ has the inverse $R_\Gamma:
N_\Gamma(\epsilon) \rightarrow [-\epsilon, \epsilon]\times [\xi_s,
\xi_e]$ such that
\begin{eqnarray}
\begin{bmatrix}
\eta \\
\xi
\end{bmatrix} =
\begin{bmatrix}
\eta(x, y) \\
\xi(x, y)
\end{bmatrix} = R_\Gamma(x, y)  = P_\Gamma^{-1}(x, y) \in
[-\epsilon, \epsilon]\times [\xi_s, \xi_e],~~\forall X = (x, y) \in
N_\Gamma(\epsilon). \label{eq:P_inv}
\end{eqnarray}
Consequently, for a mesh $\T_h$ with $h$ small enough, all the
interface elements are inside the $\epsilon$-tubular neighborhood of
$\Gamma$. Moreover, for each interface element $K \in \T_h^i$,
there exists a rectangle $\hK_F = [-h, h] \times [\xi_{1,K},
\xi_{2,K}]$ with $\xi_{1, K}, \xi_{2, K} \in [\xi_s, \xi_e]$ such
that its image $K_F = P_\Gamma(\hK_F)$ under the Frenet transformation contains
$K$. Geometrically, $K_F$ is a curved trapezoid containing $K$, with
two curved edges parallel to
the interface curve $\Gamma$ and two straight edges. By the inverse
mapping, the interface element $K$ becomes the curved quadrilateral
$\hK = R_{\Gamma}(K) \subset \hK_F$. Very critically, the interface
curve segment $\Gamma_{K_F} = \Gamma \cap K_F$ is mapped to the
vertical line segment $\hGamma_{K_F} = R_\Gamma(\Gamma_{K_F})$ and
the interface jump conditions specified in
\eqref{eqn:interface_conds} and \eqref{eqn:interface_conds_ext} become
\begin{eqnarray}
\bb{\hu}_{{\hat \Gamma}_{K_F}} = 0, ~\bb{ \hbeta \hu_\eta }_{{
\hat{\Gamma}}_{K_F}} = 0,
~\bb{ \hbeta \frac{\partial^j}{\partial \eta^j}\mathscr{L}(\hu)
}_{{\hat \Gamma}_{K_F}} = 0,~
j = 0, 1, 2, \cdots, m-2, \label{eqn:interface_conds_eta-xi}
\end{eqnarray}
where $\hbeta(\eta, \xi) = \beta(x(\eta, \xi), y(\eta, \xi))$ and
$\hu(\eta, \xi) = u(x(\eta, \xi), y(\eta, \xi))$ are defined by the
Frenet transformation \eqref{eq:P_map} and $\mathscr{L}$ is the
Laplacian with respect to the local coordinate $\eta$-$\xi$.
For every interface element $K \in \T_h$, following
\cite{adjeridHighOrderGeometry2024}, we will call
$\hK$ the Frenet interface element of $K$, $\hK_F$ the Frenet
fictitious element of $K$, $K_F$ the fictitious element of $K$, and
$\hGamma_{K_F}$ the Frenet interface of $K$, and their illustrations
are given in
\autoref{fig:intf_fic_elements}. Without loss of
generality, we have assumed that
the parametrization $\g(\xi)$ of the interface curve $\Gamma$ is
such that the normal $\n(\xi)$ points
from $\Omega^-$ to $\Omega^+$ at every point $X = \g(\xi) \in
\Gamma$. Also, we have assumed that
the value of $\hbeta$ on the right-hand side of $\hGamma_{K_F}$ is
$\beta^+$, and the value of $\hbeta$ on the left-hand side of
$\hGamma_{K_F}$ is $\beta^-$.

\begin{figure}[ht]
\centerline{
\includegraphics[scale=1]{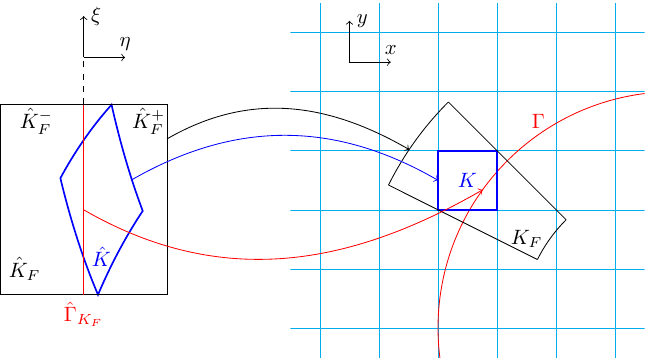}
\hspace{0.4in}
}
\caption{An interface element $K$ and its associated sets $\hK_F,
K_F$ and $\hGamma_{K_F}$.}
\label{fig:intf_fic_elements}
\end{figure}

The fact that $\hGamma_{K_F}$ is a straight line segment greatly
simplifies the construction of IFE functions on $\hK_F$ with
polynomials with a degree of our choice. In fact, it has been proved
in \cite{adjeridHighOrderGeometry2024} that the following IFE space in
local coordinates is well-defined:
\begin{eqnarray}
{\hat \V}^m_{\hbeta}(\hK_F) &=&\left\{\hphi:\hK_F \rightarrow \mathbb{R}\mid
\hphi|_{\hK_F^{\pm}} \in \mathbb{Q}_m~~\text{ and } ~\hphi \text{
satisfies \eqref{eqn:interface_conds_eta-xi}} \right\}
\label{eq:LocalIFESpace_etaxi}
\end{eqnarray}
where $\mathbb{Q}_m$ is the space of tensor-product polynomials of
degree $m$ in each variable. Consequently, the local GC-IFE space on
the interface element $K$ is defined by
\begin{eqnarray}
\V^m_{\beta}(K)=\left\{\hphi\circ(R_\Gamma)|_{K} \mid \hphi \in
{\hat \V}^m_{\hbeta}(\hK_F)\right\}.
\label{eq:LocalIFESpace_xy}
\end{eqnarray}
Reference \cite{adjeridHighOrderGeometry2024} proves that functions
in $\V^m_{\beta}(K)$ precisely satisfy the interface jump conditions
\eqref{eqn:interface_conds}, whereas all other IFE functions
previously proposed in the literature can only satisfy these
conditions approximately. Moreover,
it has been proved in \cite{adjeridHighOrderGeometry2024} that the
GC-IFE space $\V^m_{\beta}(K)$ possesses
the optimal approximation capability in the sense that the $L^2$
projection of a function $u$ unto $\V^m_{\beta}(K)$ can approximate
$u$ optimally with respect to the underlying polynomial space
$\mathbb{Q}_m$, provided that $u$ is piecewisely smooth enough and
satisfies jump
conditions \eqref{eqn:interface_conds} and \eqref{eqn:interface_conds_ext}. The
satisfaction of jump condition \eqref{eqn:interface_conds} by GC-IFE
functions in
$\V^m_{\beta}(K)$ ensures that they belong to $H^1(K)$. This property allows the
natural employment of the GC-IFE space $\V^m_{\beta}(K)$ in an
optimally convergent
IFE method based on a standard DG formulation
\cite{adjeridHighOrderGeometry2024} to solve the interface problem
\eqref{eqns:interface_problem}.

For solving the interface problem \eqref{eqns:interface_problem}, we
need a basis for
the local GC-IFE space $\V^m_{\beta}(K)$ and use functions in this
basis to assemble matrices and vectors
on each interface element $K \in \T_h^i$ as required by a chosen
finite element method.
The current article expands on the preliminary concepts introduced in
\cite{adjeridHighOrderGeometry2024} and
\cite{meghaichiHigherOrderImmersed2024} by proposing a construction
procedure for generating a GC-IFE basis. This basis not only
satisfies interface jump conditions but also improves the stability
and accuracy of approximations for interface problems. The
construction procedure to be presented consists of two stages: an
initial construction stage and a reconstruction stage. The initial
construction builds upon and extends the ideas in
\cite{adjeridHighOrderGeometry2024} and
\cite{meghaichiHigherOrderImmersed2024} to develop a GC-IFE basis
that satisfies specified jump conditions. In the reconstruction
stage, we refine the initial basis to produce an orthonormal GC-IFE basis with
respect to the $L^2$ inner product.

We will address key computations and implementations in the proposed
construction which are not common in conventional finite element
methods. Sample \MATLAB functions will be provided mainly for
illustration. Their efficiency can most probably be improved even
though the efficiency was in consideration. We believe the
discussions in this article can facilitate the adoption of the
recently introduced,
feature-rich GC-IFE functions for solving interface problems where
the partial differential equations involve the second order elliptic
operator, such as the elliptic equation in the interface problem
\eqref{eqns:interface_problem} as well as the heat and wave
equations. Furthermore, the framework and methodology outlined here
may offer insights into extending the GC-IFE approach to more complex
interface problems involving systems of PDEs, including, but not
limited to, linear elasticity, the Stokes system, and coupled
acoustic-elastic models.

The rest of this article is organized as follows.
\autoref{sec:prelim} is about essential calculus formulas in terms of the
local coordinates. While most of the calculus formulas to be
presented are from \cite{adjeridHighOrderGeometry2024} and
\cite{meghaichiHigherOrderImmersed2024}, we will discuss computation
details about the Frenet transformation between the physical
coordinates $X = (x, y)$ and the local coordinates
based on the differential geometry of the interface. Sample \MATLAB
functions are given for illustrating how to implement these
transformations.
\autoref{sec:local_space} is the core of this
article, and it is about the computational details and
implementation issues for the initial construction of a GC-IFE basis
using the procedure originally developed in
\cite{adjeridHighOrderGeometry2024} and
\cite{meghaichiHigherOrderImmersed2024}. The second part is a
generalized initial construction based on an extension idea
\cite{adjeridFrenetImmersedFinite2025,
guoDesignAnalysisApplication2019, guoHigherDegreeImmersed2019} which
forms an IFE basis function by choosing a polynomial on one side of
the interface and extending it to other side according to the jump
conditions across the interface.
The third part is about the
reconstruction procedure which transforms the initially constructed
basis into an
orthonormal basis efficiently and  two approaches for carrying out the
orthonormalization are presented.
\autoref{sec:NumExample}
presents a numerical example that demonstrates how computations
involving GC-IFE basis functions are performed, including
preprocessing and local matrix assembly.
\MATLAB functions for
key computations in the construction procedure
will be presented and explained throughout the manuscript.

\section{Calculations in local variables} \label{sec:prelim}

We can transform a function $u(x,y)$ defined in physical coordinates
$(x,y)$ into a function $\hu(\eta, \xi) = u(x(\eta, \xi), y(\eta,
\xi))$ expressed in local coordinates $(\eta, \xi)$ near the
interface $\Gamma$. These local coordinates $(\eta, \xi)$ are
intrinsically linked to the geometry of $\Gamma$, with $\eta(x, y)$,
for example, indicating the distance of $\x = (x, y)$ from $\Gamma$.
Since $(\eta, \xi)$ constitutes an orthogonal curvilinear coordinate
system in the neighborhood of $\Gamma$, standard calculus operations
with respect to these coordinates are readily available in textbooks.
However, for the reader's ease of reference and for their application
in later sections, we provide a set of essential calculus formulas in
this preliminary section. Most of the formulas to be presented are
from \cite{adjeridHighOrderGeometry2024} and
\cite{meghaichiHigherOrderImmersed2024}.

\subsection{Calculus in terms of the local coordinates}

We assume that the interface curve $\Gamma$ has a regular parametrization:
\begin{eqnarray}
\g=
\begin{bmatrix}
g_1 \\
g_2
\end{bmatrix}:[\xi_s,\xi_e]\to \mathbb{R}^2. \label{eq:g_formula}
\end{eqnarray}
Then, the Frenet apparatus of $\Gamma$ is  defined by
\begin{eqnarray}
&&\ttau(\xi) = \frac{1}{\norm{\g'(\xi)}} \g'(\xi), ~\n(\xi) =
Q\ttau(\xi) \text{~~with~~} Q =
\begin{bmatrix}
0 & 1 \\
-1 & 0
\end{bmatrix}, \label{eq:tangent_normal_vectors} \\
&&\kappa(\xi)=\norm{\g'(\xi)}^{-3}\det\big[\g'(\xi),\ \g''(\xi)\big].
\label{eq:curvature}
\end{eqnarray}
By the well-known Frenet-Serret formulas
\cite{abbenaModernDifferentialGeometry2006} for the tangent
and normal vectors:
\begin{equation}
\ttau'(\xi)=-\kappa(\xi)\norm{\g'(\xi)}\n(\xi),\qquad
\n'(\xi)=\kappa(\xi)\norm{\g'(\xi)}\ttau(\xi).
\label{eqn:frenet-serret}
\end{equation}
By \eqref{eq:tangent_normal_vectors} and \eqref{eqn:frenet-serret},
we can derive the following formula for the Jacobian of the Frenet
transformation $P_\Gamma(\eta, \xi)$:
\begin{eqnarray}
J_{P_\Gamma}(\eta, \xi) &=&
\begin{bmatrix} \n(\xi), ~~\g'(\xi)+\eta\n'(\xi)
\end{bmatrix}
=
\begin{bmatrix} \n(\xi), ~~\norm{\g'(\xi)}\left(1+\eta \kappa(\xi)
\right)\ttau(\xi)
\end{bmatrix}. \label{eq:Jacobian_P_map}
\end{eqnarray}
Then, by the inverse function theorem, we have the following formula
for the Jacobian of the inverse
Frenet transformation
\begin{eqnarray}
J_{R_\Gamma}(x, y) &=&
\begin{bmatrix}
\n(\xi)^T\\ \rho(\eta, \xi)\ttau(\xi)^T
\end{bmatrix}, \label{eq:Jacobian_R_map}
\end{eqnarray}
with
\begin{eqnarray}
\rho(\eta, \xi) = \norm{\g'(\xi)}^{-1}\psi(\eta, \xi), ~~\psi(\eta,
\xi) = \frac{1}{1+\eta \kappa(\xi)}.
\label{eq:psi_rho}
\end{eqnarray}
By \eqref{eq:P_inv} and the definition of $J_{R_\Gamma}(x, y)$ we can see that
\begin{equation}
\nabla \eta(\x) =\n(\xi),\qquad \nabla \xi(\x) =
\norm{\g'(\xi)}^{-1}\psi(\eta,\xi)\ttau(\xi),
~~\forall ~ { \x} \in N_\Gamma(\epsilon). \label{eqn:grad_eta_xi}
\end{equation}
By \eqref{eqn:grad_eta_xi} and the chain rule, we have the following
formula for the gradient of
$u(\x) = u(x, y) = \hu(\eta(x, y), \xi(x, y))$:
\begin{eqnarray}
\nabla u(\x) &=& \hu_\eta(\eta(x, y), \xi(x, y)) \nabla \eta(x, y)
+ \hu_\xi(\eta(x, y), \xi(x, y)) \nabla \xi(x, y)  \nonumber \\
&=& \hu_{\eta}(\eta,\xi) \n(\xi) + \hu_{\xi}(\eta,\xi)\rho(\eta,
\xi)\ttau(\xi), \label{eqn:grad_u}
\end{eqnarray}
which expresses the gradient $\nabla u(\x)$ in terms of the normal
vector $\n(\xi)$ and the tangential vector $\ttau(\xi)$.
\noindent
For the divergence of the normal vector $\n(\xi)$ and the tangential
vector $\ttau(\xi)$, we have
\begin{eqnarray}
\nabla \cdot \n(\xi(\x)) &=& \kappa(\xi)\psi(\eta, \xi),
~\nabla\cdot \ttau(\xi(\x)) = 0.
\label{eqn:div_eta_xi}
\end{eqnarray}
For a generic vector function $\mathbf{f}: N_\Gamma(\epsilon) \to
\mathbb{R}^2$ expressed in
terms of $\n(\xi)$ and $\ttau(\xi)$
\begin{eqnarray*}
\mathbf{f}(\x) = f^{\n}(\x)\n(\xi(\x))
+f^{\ttau}(\x)\ttau(\xi(\x)),~~\forall ~{ \x} \in N_\Gamma(\epsilon),
\end{eqnarray*}
by \eqref{eqn:grad_u} and \eqref{eqn:div_eta_xi}, we have
\begin{eqnarray}
\nabla \cdot \mathbf{f}(\x) &=& (\nabla f^{\n}(\x))\n(\xi(\x)) +
f^{\n}(\x)(\nabla \cdot \n(\x))
+ (\nabla f^{\ttau}(\x))\ttau(\xi(\x)) + f^{\ttau}(\x)(\nabla \cdot
\ttau(\x))  \nonumber \\
&=& \hat{f}^{\n}_{\eta}(\eta,\xi) + \kappa(\xi)\psi(\eta,\xi)
\hat{f}^{\n}(\eta,\xi) + \rho(\eta, \xi)\hat{f}^{\ttau}_{\xi}(\eta,\xi),
~~\forall ~{ \x} \in N_\Gamma(\epsilon).
\label{eqn:div_f}
\end{eqnarray}
Consequently, by \eqref{eqn:grad_u} and \eqref{eqn:div_f}, we have
the following formula for
$\Delta u(\x)$:
\begin{subequations}  \label{eqn:laplacian_form}
\begin{eqnarray}
\Delta u(\x) &=& \nabla\cdot (\nabla u(\x)) = \L(\hu(\eta, \xi))
~~\text{with}\nonumber \\
\L(\hu(\eta, \xi)) &:=& \hu_{\eta\eta}(\eta,\xi)+ J_0(\eta,\xi)
\hu_{\xi\xi}(\eta,\xi) + J_1(\eta,\xi)\hu_{\eta}(\eta,\xi) +
J_2(\eta,\xi)\hu_{\xi}(\eta,\xi), \label{eqn:laplacian_formula}
\end{eqnarray}
where
\begin{equation}
\begin{split}
J_0(\eta,\xi) &=\rho^2(\eta, \xi),\qquad J_1(\eta,\xi)
=\kappa(\xi)\psi(\eta,\xi), \\
J_2(\eta,\xi) &=-\rho^2(\eta, \xi)\left(\eta
\kappa'(\xi)\psi(\eta,\xi)+\frac{\g'(\xi)\cdot
\g''(\xi)}{\norm{\g'(\xi)}^{2}}\right),
\end{split} \label{eqn:def_of_Js}
\end{equation}
\end{subequations}
and, hereinafter, we tacitly assume that the parametrization
$\g(\xi)$ of the interface curve $\Gamma$ is $C^3$ whenever the term
$\L(\hu(\eta, \xi))$ is involved.

We now consider the jump conditions in local coordinates $(\eta,
\xi)$. By the Frenet transformation $\hu(\eta, \xi) = u(x(\eta, \xi),
y(\eta, \xi))$, we can see that the first jump condition in
\eqref{eqn:interface_conds} leads to
\begin{eqnarray*}
\bb{\hu}_{{\hat \Gamma}_{K_F}} = \bb{u}_{\Gamma_{K_F}} = 0.
\end{eqnarray*}
From \eqref{eqn:grad_u}, we have $\nabla u(\x) \cdot \n =
\hu_{\eta}(\eta,\xi)$; hence, the second jump condition implies
\begin{eqnarray*}
\bb{ \hbeta \hu_\eta }_{{\hat \Gamma}_{K_F}} = \bb{\beta \nabla
u\cdot \n}_{\Gamma_{K_F}}=0,
\end{eqnarray*}
and we note that the normal derivative $\nabla u\cdot \n$ along the
curve $\Gamma$ is transformed to
the simpler partial derivative $\hu_\eta$ on the $\xi$-axis.
Similarly, by the gradient operation given in \eqref{eqn:grad_u} and
the Laplacian in  \eqref{eqn:laplacian_form} we can see that the
extended jump condition \eqref{eqn:interface_conds_ext} suggests
\begin{eqnarray*}
\bb{ \hbeta \frac{\partial^j}{\partial \eta^j}\mathscr{L}(\hu)
}_{{\hat \Gamma}_{K_F}} =
\bb{\beta \p_{\n^j}\Delta u}_{\Gamma_{K_F}} = 0, ~ j = 0, 1, 2, \ldots, m.
\end{eqnarray*}
Therefore, in the Frenet fictitious element $\hK_F$, $u(\x) =
\hu(\eta, \xi)$ should satisfy the jump conditions given in
\eqref{eqn:interface_conds_eta-xi}.

The construction of the GC-IFE basis functions involves computations
specified by the jump conditions. Because ${\hat \Gamma}_{K_F}$ is a
line segment along the $\xi$-axis in the local coordinate system,
the computations for the first two jump conditions in
\eqref{eqn:interface_conds_eta-xi} are trivial:
\begin{eqnarray*}
\bb{\hu}_{{\hat \Gamma}_{K_F}} = \hu^+(0, \xi) - \hu^-(0, \xi), ~~
\bb{ \hbeta \hu_\eta }_{{\hat \Gamma}_{K_F}} = \hbeta^+
\hu_\eta^+(0, \xi) - \hbeta^- \hu_\eta^-(0, \xi),
~(0, \xi) \in {\hat \Gamma}_{K_F},
\end{eqnarray*}
with $\hu^s = \hu|_{\hK_F^s}, ~s = \pm$, see the illustration in
\autoref{fig:intf_fic_elements}.
Computations for the more complex extended jump conditions benefits
even more from the simple geometry and the location of the
transformed interface ${\hat \Gamma}_{K_F}$. First,
\begin{eqnarray*}
\bb{ \hbeta \frac{\partial^j}{\partial \eta^j}\mathscr{L}(\hu)
}_{{\hat \Gamma}_{K_F}} =
\hbeta^+ \frac{\partial^j}{\partial \eta^j}\mathscr{L}(\hu^+(0, \xi)) -
\hbeta^- \frac{\partial^j}{\partial \eta^j}\mathscr{L}(\hu^-(0, \xi))
\end{eqnarray*}
with
\begin{eqnarray}
\begin{aligned} \label{eq:eta_d_L}
&\frac{\partial^j}{\partial \eta^j} \L(\hu^s)(0,\xi) =
\frac{\partial^{j+2}\hu^s(0,\xi)}{\partial\eta^{j+2}} \\
& +\sum_{i=0}^{j}\binom{j}{i}\left(
\frac{\partial^{i}J_0(0,\xi)}{\partial\eta^{i}}
\frac{\partial^{j-i} \hu^s_{\xi\xi}(0,\xi)}{\partial\eta^{j-i}}+
\frac{\partial^{i}J_1(0,\xi)}{\partial\eta^{i}}
\frac{\partial^{j-i+1}\hu^s(0,\xi)}{\partial\eta^{j-i+1}}+
\frac{\partial^{i}J_2(0,\xi)}{\partial\eta^{i}}
\frac{\partial^{j-i} \hu^s_{\xi}(0,\xi)}{\partial\eta^{j-i}}\right)
\end{aligned}, ~~~s = \pm.
\end{eqnarray}
In the construction of GC-IFE basis functions, $\hu^s$ in
\eqref{eq:eta_d_L} will be polynomials chosen form $Q_m$; hence,
their derivatives
\begin{eqnarray*}
\frac{\partial^{j+2}\hu^s(0,\xi)}{\partial\eta^{j+2}},
~\frac{\partial^{j-l} \hu^s_{\xi\xi}(0,\xi)}{\partial\eta^{j-l}},
~\frac{\partial^{j-l+1}\hu^s(0,\xi)}{\partial\eta^{j-l+1}},
~\frac{\partial^{j-l} \hu^s_\xi(0,\xi)}{\partial\eta^{j-l}}
\end{eqnarray*}
can be easily prepared. By \eqref{eq:psi_rho} and
\eqref{eqn:def_of_Js}, the $\eta$-partial derivatives
of $J_i(\eta, \xi), ~i = 0, 1, 2$ at $\eta =0$ depends on the
$\eta$-partial derivatives of
$\psi^i(\eta, \xi), ~i = 1, 2, 3$ evaluated at $\eta = 0$, which can
be simply expressed in terms of the
curvature because of the applicability of their Maclaurin expansions
with respect to $\eta$. Consequently, we have the following explicit formulas:
\begin{eqnarray}
\begin{aligned}\label{eq:J_derivatives}
&\frac{\partial^l J_0(0, \xi)}{\partial \eta^l} = \frac{(-1)^l
(l+1)!\kappa(\xi)^l}{\norm{\g'(\xi)}^2}, \quad \frac{\partial^l
J_1(0, \xi)}{\partial \eta^l} = \kappa(\xi)(-1)^l l!\kappa(\xi)^l, \\
&\frac{\partial^l J_2(0, \xi)}{\partial \eta^l} =
-\frac{\kappa'(\xi)}{\norm{\g'(\xi)}^2}\left(
l (-1)^{l-1} \frac{(l+1)!}{2}\kappa(\xi)^{l-1} \right) -
\frac{\g'(\xi)\cdot \g''(\xi)}{\norm{\g'(\xi)}^4}(-1)^l (l+1)!\kappa(\xi)^l,
\end{aligned}
\end{eqnarray}
for the evaluation of $\frac{\partial^j}{\partial \eta^j}
\L(\hu^s)(0,\xi), ~s = \pm$ according to \eqref{eq:eta_d_L}. In other
words, because of the simple geometry and the location of ${
\hat\Gamma}_{K_F}$, we have an explicit formula for evaluating the
transformed extended jump conditions. Readers can refer to
\cite{adjeridHighOrderGeometry2024} and
\cite{meghaichiHigherOrderImmersed2024} for details about the
derivation of these formulas.

\subsection{The Frenet transformations between the physical and local
coordinates}\label{sec:Frenet_transf}

In this section, we discuss the computation and implementation for
the Frenet transformation $P_\Gamma(\eta, \xi)$ defined by
\eqref{eq:P_map} and the inverse of the Frenet transformation
$R_\Gamma(\x)$ defined by \eqref{eq:P_inv}. First, the computation
and implementation for the Frenet transformation $P_\Gamma(\eta, \xi)
= \x$ is straightforward once the parametrization $\g(\xi)$ for the
interface curve $\Gamma$ and the normal $\n(\xi)$ have been provided.
The program \verb!FrenetPmap.m! given in
\cite{adjeridMATLABImplementationGeometryConforming} is a sample
implementation of $P_\Gamma(\eta, \xi)$ in \MATLAB which can be used
as follows:
\begin{matlab}
xyList = FrenetPmap(etaxiList, CurveDG)
\end{matlab}
For the inputs, \verb!etaxiList! is an array with two rows such that
each of its columns provides the local coordinates of a point, while
\verb!CurveDG! is a structure variable with fields associated with
the differential geometry of the curve $\Gamma$ such that \verb!g!
for $\g(\xi)$, \verb!gp! for $\g'(\xi)$, \verb!gpp! for $\g''(\xi)$,
\verb!gt! for $\ttau(\xi)$, \verb!gn! for
$\n(\xi)$, \verb!curvature! for $\kappa(\xi)$, \verb!curvaturep! for
$\kappa'(\xi)$, and \verb!xi_domain! is an array formed by the end
points of the interval $[\xi_s, \xi_e]$. The output \verb!xyList! is an
array of the same size as \verb!etaxiList! whose columns contain the
physical coordinates of the points whose corresponding local
coordinates are specified by \verb!etaxiList!.

As indicated by equation \eqref{eq:LocalIFESpace_xy}, the inverse
Frenet transformation $R_\Gamma(\x) = (\eta, \xi)$ is critical for
constructing GC-IFE functions. This inverse transformation
is generally fully nonlinear due to the nonlinearity inherent to the
geometry of the interface curve $\Gamma$. The absence of an explicit
formula for computing $(\eta(\x), \xi(\x))$ directly from the
parametrization $\g(\xi)$ and the Frenet apparatus of $\Gamma$
necessitates a numerical approach to solve equation \eqref{eq:P_map}
for $(\eta, \xi)$ given a physical coordinate $\x$. Since the
Jacobian of equation \eqref{eq:P_map} is given explicitly by
\eqref{eq:Jacobian_P_map} in terms of the Frenet apparatus, we
propose employing Newton's method to implement the inverse Frenet
transformation $R_\Gamma(\x) = (\eta, \xi)$.

Given an interface element $K$ and a point $\x \in K$, an initial guess
$(\eta_{0,K}, \xi_{0, K})$ is needed for computing $(\eta, \xi)$,
the solution to equation \eqref{eq:P_map}, using Newton's method.
As described in \cite{adjeridHighOrderGeometry2024}, as
well as by \eqref{eq:LocalIFESpace_xy} in the previous section, the
construction of GC-IFE functions is carried out on each interface
element $K \in \T_h^i$ and the inverse Frenet transformation $R_\Gamma$ is
applied to $\x \in K$. This implies that $(\eta, \xi) = R_\Gamma(\x)$
is in the Frenet fictitious element $\hK_F = [-h, h] \times
[\xi_{1,K}, \xi_{2,K}]$.
Therefore, when the mesh size $h$ of $\T_h$ is small, $\eta$ should
be close to $0$ and $\eta_{0, K} = 0$ should be a good initial guess
for a Newton iteration. We emphasized that this is a good initial
guess for $\eta$ for every
interface element $K \in \T_h^i$ because of the nature of the Frenet
fictitious element $\hK_F = [-h, h] \times [\xi_{1,K}, \xi_{2,K}]$.

Again, since the construction of GC-IFE functions needs to be carried
out on all interface elements,
we propose the following direct search approach for preparing the
initial guess for $\xi$ for a Newton iteration that computes
$R_\Gamma(\x)$ on each  interface element. First, we sample a
sufficient number of points
$\{\g(\xi^{(i)})\}_{i=1}^{\mathtt{N}_{\g}}$ with
$\{\xi^{(i)}\}_{i=1}^{\mathtt{N}_{\g}} \subset [\xi_s,\xi_e]$ on the
interface curve $\Gamma$,
see these sampled points on $\Gamma$ in the illustration in
\autoref{fig:sample_points}. For each interface element $K \in \T_h^i$ we
let $K_c$ be its center, and then, we let $\xi_{0,K}$ be such that
\begin{eqnarray}
{\xi}_{0,K} \in \{\xi^{(i)}\}_{i=1}^{\mathtt{N}_{\g}} ~~\text{and}~~
{ \xi}_{0,K} = \argminA_{1 \leq i \leq \mathtt{N}_{\g}}
\norm{\g\left(\xi^{(i)}\right)-K_c}. \label{eqn:xi_guess}
\end{eqnarray}
An efficient method based on the nearest point search can be used to
implement \eqref{eqn:xi_guess} when
$\mathtt{N}_{\g}$ is moderately large, and in general, we let
$\mathtt{N}_{\g} = O(1/h)$ for a mesh
$\T_h$.

\begin{figure}[htbp]
\center
\includegraphics[scale=1.3]{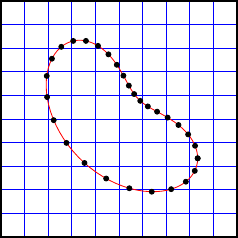}
\caption{The sample points $\g(\xi^{(i)})$ (black) on the interface
$\Gamma$ (red)}
\label{fig:sample_points}
\end{figure}

The program \verb!xiInitGuess.m! given in
\cite{adjeridMATLABImplementationGeometryConforming} is a sample
implementation of the procedure above for generating an initial guess
${ \xi}_{0,K}$ for each interface element $K \in \T_h^i$. The
usage of this function is as follows:
\begin{matlab}
xiGuessList = xiInitGuess(mesh, IntElemList, CurveDG)
\end{matlab}
The input argument \verb!mesh! is a structure array defining a mesh
$\T_h$, whose field \verb!mesh.t! defines the mesh elements, and the
field \verb!mesh.p! defines mesh points. {\verb!IntElemList!} is an
array of integers
for the indices of interface elements in the mesh $\T_h$. The input
\verb!CurveDG!, as described before, provides functions associated
with the differential geometry of the interface curve $\Gamma$.
The output \verb!xiGuessList! is an array of the same length as
\verb!IntElemList! such that
the $i$-th entry of \verb!xiGuessList! is the value of
${\tilde\xi}_{0,K}$ determined for the $i$-th interface element
$K$ listed by
\verb!IntElemList!.

We note that alternative methods exist for determining a suitable
initial guess $\xi_{0, K}$. For example,
\cite{adjeridHighOrderGeometry2024} employs a gradient descent
method. In general, this approach could supersede the direct search
\eqref{eqn:xi_guess} and may be particularly advantageous for an
interface $\Gamma$ with more complicated geometry. However, in our
numerical experiments, we found no significant difference, as
Newton's iteration consistently converged rapidly (usually in fewer
than 10 iterations) regardless of which method prepared $\xi_{0, K}$.

The usual implementation of Newton's method involves solving a linear
system at each iteration, whose coefficient matrix is the Jacobian
matrix of the system of nonlinear equations to be solved. {However,
when we use it to numerically solve $P_\Gamma(\eta, \xi) - \x = 0$
for $(\eta, \xi)$ for the inverse Frenet transformation $R_\Gamma(\x) =
(\eta, \xi)$, the Jacobian matrix is $J_{P_\Gamma}(\eta, \xi)$, and
its inverse is
$J_{R_\Gamma}(\eta, \xi)$ { which can be} directly computed by the Frenet
apparatus of $\Gamma$ according to
\eqref{eq:Jacobian_R_map}.} This suggests adopting the following
Newton iteration for implementing
the inverse Frenet transformation $R_\Gamma(\x) = (\eta, \xi)$:
\begin{enumerate}
\item Set $(\eta_0,\xi_0)=(0,\xi_{0,K})$.
\item Update the guess using a Newton iteration
\begin{equation}
\begin{bmatrix}
\eta_{i+1} \\
\xi_{i+1}
\end{bmatrix} =
\begin{bmatrix}
\eta_{i} \\
\xi_{i}
\end{bmatrix} -
\begin{bmatrix}
\n(\xi_i)^T\\ \rho(\eta_i, \xi_i)\ttau(\xi_i)^T
\end{bmatrix}(P_{\Gamma}(\eta_i, \xi_i) - \x).
\label{eqn:Newton}
\end{equation}
\end{enumerate}
We note that the iteration is carried out by the evaluation of the
Frenet apparatus of $\Gamma$ and without the need to solve a linear system.

A sample implementation of the inverse Frenet transformation is the
\MATLAB function \verb!FrenetRmap.m! in
\cite{adjeridMATLABImplementationGeometryConforming}, and it can be
used as follows:
\begin{matlab}
etaxiList = FrenetRmap(xyList, CurveDG, xiGuess, maxIter)
\end{matlab}
For the inputs, \verb!xyList! an array with two rows such that each
of its columns provides the physical  coordinates of a point in an
interface element $K \in \T_h^i$, \verb!xiGuess! is a guess about
$\xi$ associated with points $(x, y) \in K$, which can be ${
\tilde{\xi}}_{0, K}$ in \eqref{eqn:xi_guess} or
$\ximid$ to be defined in \eqref{eq:xi_01_etah_xih}, and
\verb!maxIter! is the maximum number of iterations allowed for the
Newton iteration \eqref{eqn:Newton} to be carried out.

\section{Construction of a basis for the GC-IFE space}\label{sec:local_space}

In this section, we detail the key computations involved in
constructing a basis for the GC-IFE space. Our work extends the
essential idea of GC-IFE functions, first introduced in
\cite{adjeridHighOrderGeometry2024}, by presenting a novel
construction procedure designed to yield a basis capable of providing
better approximation for interface problems. The construction
procedure to be presented consists of two stages: an initial
construction stage and a reconstruction stage. For initial
construction stage we first provide details for the computational and
implementation aspects of the basis construction procedure outlined
in \cite{adjeridHighOrderGeometry2024}. Then, we introduce a
generalized initial construction procedure that extends a chosen
polynomial from one side of the interface $\Gamma$ to the other side,
adhering to specified jump conditions. This approach is inspired by
the work in \cite{adjeridFrenetImmersedFinite2025,
guoDesignAnalysisApplication2019, guoHigherDegreeImmersed2019}.
This general construction procedure not only facilitates the
construction of local matrices associated with the produced GC-IFE
basis on interface elements, but it also serves as a crucial
preparatory step for the reconstruction stage.
The reconstruction stage is motivated by the desire to create a
GC-IFE basis with
enhancement that ensures stable and accurate computations for
interface problems.

\subsection{An initial construction of a GC-IFE basis} \label{sec:method_1}
We start from a suitable choice of basis for the underlying
polynomial space $\mathbb{Q}_m$ which is used to construct GC-IFE
functions in a piecewise polynomial form. Let $\mathbb{P}_m$ be the
space of polynomials of one variable with degree not exceeding $m$. Let
$\{\hp_i(x)\}_{i = 0}^m$ and $\{\hq_i(x)\}_{i = 0}^m$ be two bases of
$\mathbb{P}_m$ such that $deg(\hp_i) = deg(\hq_i) = i, ~0 \leq i \leq
m$, and let
\begin{eqnarray}
\phi_{i,j}(\eta,\xi)=\hq_j(\eta)\hp_i(\xi), ~~0 \leq i, j \leq m.
\label{eqn:basis_R}
\end{eqnarray}
Then, $\mathbb{Q}_m = \operatorname{span}\{\phi_{i,j}(\eta,\xi), ~0
\leq i, j \leq m\}$.
However, it is desirable to choose the $\hq$-polynomials such that
$\hq_j(0) = 0, ~1 \leq j \leq m$ and $\hq_j'(0)=0,~ 2\le j\le m$, in
order to take the advantage of
the special location and geometry of the Frenet interface
$\hGamma_{K_F}$ in the transformed jump conditions
\eqref{eqn:interface_conds_eta-xi} imposed on the Frenet fictitious
element $\hK_F$. For example, \cite{adjeridHighOrderGeometry2024} and
\cite{meghaichiHigherOrderImmersed2024} used
the canonical monomials $\hq_j(\eta) = \eta^j, ~0 \leq j \leq m$ to
establish the existence of GC-IFE ${\hat \V}^m_{\hbeta}(\hK_F)$ on
$\hK_F$. The authors suggested other choices in
\cite{adjeridHighOrderGeometry2024} and
\cite{meghaichiHigherOrderImmersed2024} for $\{\hp_i(x)\}_{i = 0}^m$
and $\{\hq_i(x)\}_{i = 0}^m$ based on other important considerations
such as computational cost and stability. To be specific, for each
interface element $K = \square A_1A_2A_3A_4$, we let
\begin{align}
\begin{split}
\eta_h &= \max\{\abs{\eta_i}, ~(\eta_i, \xi_i) = R_\Gamma(A_i),
~1 \leq i \leq 4\}, \\
\xi_{0, K} &= \min\{\xi_i, ~(\eta_i, \xi_i) = R_\Gamma(A_i), ~1
\leq i \leq 4\}, \\
\xi_{1, K} &= \max\{\xi_i, ~(\eta_i, \xi_i) = R_\Gamma(A_i), ~1
\leq i \leq 4\}, \\
\ximid &= \frac{1}{2}(\xi_{1, K} + \xi_{0, K}), ~\xi_h =
\frac{1}{2}(\xi_{1, K} - \xi_{0, K}).
\end{split} \label{eq:xi_01_etah_xih}
\end{align}
Then, we propose to choose $\{\hp_i(\xi)\}_{i = 0}^m$ such that
\begin{eqnarray}
\hp_i(\xi) = p_{i}\left(\frac{\xi-\ximid}{\xi_h}\right),
~\hq_i(\eta) = q_i\left(\frac{\eta}{\eta_h}\right), ~0 \leq i \leq
m, \label{eq:pi-qi_polynomials}
\end{eqnarray}
in which $\{p_i\}_{i = 0}^m$ are the standard Legendre polynomials
\cite{abramowitz1964handbook}, $q_0(x) = 1,
q_1(x) = p_1(x)$ and
\begin{eqnarray}
\begin{aligned}
q_i(x) &= ip_{i}(x) + (i-1)p_{i-2}(x) - (2i-1)p_{i-1}(0)x \\
& = (2i-1)xp_{i-1}(x) - (2i-1)p_{i-1}(0)x
\end{aligned}, \hspace{0.3in}2 \leq i \leq m. \label{eq:q_poly}
\end{eqnarray}
Then we can verify that these $q$-polynomials have the desirable
property: $q_i(0) = q_i'(0) = 0, ~2 \leq i \leq m$ and $q_1(0)=0$.

As proved in \cite{adjeridHighOrderGeometry2024}, the dimension of
the GC-IFE space ${\hat \V}^m_{\hbeta}(\hK_F)$ is $(m+1)^2$. This
motivates us to search  for a set of basis functions
in the following form:
\begin{equation}
\hphi_{i,j}(\eta,\xi)=
\begin{cases}
\hphi_{i,j}^+(\eta,\xi), & (\eta,\xi)\in \hK^+_F,\\
\hphi_{i,j}^-(\eta,\xi), & (\eta,\xi) \in \hK^-_F,
\end{cases} \hspace{0.1in} 0 \leq i, j \leq m, ~~ \hphi_{i,j}^+,
~\hphi_{i,j}^- \in \mathbb{Q}_m.
\label{eqn:hphi_Kp_Km_ij}
\end{equation}
Following the same arguments as for Lemma 3 in
\cite{adjeridHighOrderGeometry2024}, we can show that
the following functions
\begin{equation}
\hphi_{i,j}(\eta,\xi)=
\frac{1}{\hbeta(\eta,\xi)}q_j\left(\frac{\eta}{\eta_h}
\right)p_{i}\left(\frac{\xi-\ximid}{\xi_h}\right),\qquad  \ 1\le
j\le m,\  0\le i\le m, \label{eqn:hphi_Kp_Km_ij>1}
\end{equation}
satisfy the jump conditions \eqref{eqn:interface_conds_eta-xi}, and
they are linearly independent. Hence, these functions can be used as
part of a basis for the GC-IFE space ${\hat \V}^m_{\hbeta}(\hK_F)$.

We proceed to the construction of an additional $m+1$ functions,
$\hphi_{i, 0}(\eta, \xi),~ 0 \leq i \leq m$, to complete the basis.
We first note that the jump conditions
\eqref{eqn:interface_conds_eta-xi} in local coordinates have the
following form:
\begin{subequations} \label{eqns:IFE_conditions}
\begin{equation}\label{eqn:IFE_continuity}
\hu^{+}(0,\xi)=\hu^{-}(0,\xi),\qquad \xi \in \hGamma_{K_F},
\end{equation}
\begin{equation}\label{eqn:IFE_flux}
\beta^+ \hu^{+}_{\eta}(0,\xi)=\beta^-
\hu^{-}_{\eta}(0,\xi),\qquad \xi \in \hGamma_{K_F},
\end{equation}
\begin{equation}\label{eqn:IFE_extended}
\beta^+ \int_{\xi_{0, K}}^{\xi_{1,
K}}\frac{\p^j}{\p\eta^j}\L\left(\hu^{+}\right)(0,\xi)v(\xi)d\xi=\beta^-
\int_{\xi_{0, K}}^{\xi_{1,
K}}\frac{\p^j}{\p\eta^j}\L\left(\hu^{-}\right)(0,\xi)v(\xi)d\xi,
\hspace{0.1in}\forall v\in \mathbb{P}^m([\xi_0,\xi_1]),\ 0\le j\le m-2.
\end{equation}
\end{subequations}
Also, we let
\begin{eqnarray}
\N_l(\eta,\xi) = q_t\left(\frac{\eta}{\eta_h}
\right)p_{s}\left(\frac{\xi-\xi^{\mathtt{mid}}}{\xi_h}\right),\hspace{0.1in}l
= (t-2)(m+1) + (s+1), \hspace{0.1in} \ 0\le s\le m,\  2\le t\le m.
\label{eqn:lexicographical}
\end{eqnarray}
The authors proposed the following formula for $\hphi_{i, 0}(\eta,
\xi)$ in \cite{adjeridHighOrderGeometry2024}:
\begin{equation}
\hphi_{i,0}(\eta,\xi) =
\begin{cases}
\hphi_{i,0}^-(\eta,\xi) =
p_i\left(\frac{\xi-\xi^{\mathtt{mid}}}{\xi_h}\right), & (\eta,
\xi) \in   (\eta,\xi)\in \hK^-_F,\\[5pt]
\hphi_{i,0}^+(\eta,\xi) =
p_i\left(\frac{\xi-\xi^{\mathtt{mid}}}{\xi_h}\right)+
\sum_{l=1}^{m^2-1}c^{(i)}_l \N_l(\eta,\xi), &
(\eta, \xi) \in   (\eta,\xi)\in \hK^+_F,
\end{cases} \hspace{0.1in} 0 \leq i \leq m,
\label{eqn:phi_i0}
\end{equation}
with the coefficients $c^{(i)}_l,~1 \leq l \leq m^2-1$, to be
determined such that $\hphi_{i, 0}(\eta, \xi),~ 0 \leq i \leq m$
satisfy the weak jump conditions \eqref{eqns:IFE_conditions}. Because
the $q$-polynomials are chosen such that $q_j(0) = q_j'(0) = 0, ~2
\leq j \leq m$, the basis functions
$\hphi_{i,0}(\eta,\xi), ~0 \leq i \leq m$ proposed in
\eqref{eqn:phi_i0} satisfy the jump conditions
\eqref{eqn:IFE_continuity} and \eqref{eqn:IFE_flux} for any
coefficients $c^{(i)}_l,~1 \leq l \leq m^2-1$.
This suggests to use equations \eqref{eqn:IFE_extended} to determine
the coefficients $c^{(i)}_l,~1 \leq l \leq m^2-1$. Replacing $u^+,
u^-$ and $v$ in \eqref{eqn:IFE_extended} by
$\hphi_{i,0}^+$, $\hphi_{i,0}^-$, and $\hp_{k-1},~1 \leq k \leq m+1$,
respectively, leads to a system of linear equations for the
coefficients $\c^{(i)} = \big(c^{(i)}_l\big)_{l = 1}^{m^2-1}$, which
can be written in matrix form as follows:
\begin{eqnarray}
&&\mathbf{A}\c^{(i)} =\frac{\beta^- - \beta^+}{\beta^+}
\mathbf{b}(i), ~~~~0 \leq i \leq m, ~~\text{with}~~ \mathbf{A}=
\begin{bmatrix}
A^{(0)} \\ A^{(1)} \\ \vdots\\ A^{(m-2)}
\end{bmatrix},~~ \mathbf{b}(i) =
\begin{bmatrix}
\mathbf{b}^{(0)}(i) \\
\mathbf{b}^{(1)}(i) \\
\vdots \\
\mathbf{b}^{(m-2)}(i)
\end{bmatrix}, \label{eq:Abc}
\end{eqnarray}
where
\begin{eqnarray}
&&\left(A^{(j)}\right)_{k,l}=\int_{{\hat \Gamma}_{K_F}}\left(
\frac{\partial^j}{\partial \eta^j} \mathscr{L}(\N_l)(0,\xi)
\right)\hp_{k-1}(\xi) d\xi,\qquad \text{for } ~~0 \leq j \leq m-2,
\label{eq:Aj} \\[5pt]
&&\left(\mathbf{b}^{(j)}(i)\right)_{k}=   \int_{{\hat \Gamma}_{K_F}}
\left(\frac{\partial^j}{\partial \eta^j}
\mathscr{L}(\hp_i)(0,\xi)\right)\hp_{k-1}(\xi) d\xi,
\qquad \text{for} ~~0 \leq j \leq m-2.
\label{eq:bji}
\end{eqnarray}
In general, we need to use numerical quadratures to compute the
matrix $\mathbf{A}$ and the vector $\mathbf{b}(i)$. For that,
let $\{(\czeta_r,\cw_r)\}_{r=1}^{\Nq}$ be the nodes and weights of
the Gauss-Legendre quadrature (GLQ)
on $[-1,1]$ with $\Nq\ge m+1$. Then
\begin{eqnarray}
\zeta_r=\ximid+\xi_h \czeta_r, ~~w_i=\xi_h\cw_r, ~1 \leq r \leq \Nq,
\end{eqnarray}
are the nodes and weights for the GLQ on $[\xi_{0, K}, \xi_{1, K}]$.
To  evaluate  an entry of $\mathbf{b}^{(j)}(i)$, we let
$\hu^s(\eta, \xi) = \hp_i(\xi)$ in \eqref{eq:eta_d_L} to obtain
\begin{eqnarray*}
\frac{\partial^j}{\partial \eta^j} \mathscr{L}(\hp_i)(0,\xi) &=&
\sum_{l=0}^{j}\binom{j}{l}\left(
\frac{\partial^{l}J_0(0,\xi)}{\partial\eta^{l}}
\frac{\partial^{j-l} \hp_i''(\xi)}{\partial\eta^{j-l}} +
\frac{\partial^{l}J_2(0,\xi)}{\partial\eta^{l}}
\frac{\partial^{j-l} \hp_i'(\xi)}{\partial\eta^{j-l}}\right) \\
&=& \frac{\partial^{j}J_0(0,\xi)}{\partial\eta^{j}}
\hp_i''(\xi) + \frac{\partial^{j}J_2(0,\xi)}{\partial\eta^{j}}\hp_i'(\xi).
\end{eqnarray*}
Hence, by applying the GLQ, we have
\begin{eqnarray}
\left(\mathbf{b}^{(j)}(i)\right)_k &=& \int_{{\hat \Gamma}_{K_F}}
\left(\frac{\partial^j}{\partial \eta^j}
\mathscr{L}(\hp_i)(0,\xi)\right)\hp_{k-1}(\xi) d\xi \nonumber \\
&=& \int_{\xi_{0, K}}^{\xi_{1, K}}
\frac{\partial^{j}J_0(0,\xi)}{\partial\eta^{j}}
\hp_i''(\xi) \hp_{k-1}(\xi) d\xi + \int_{\xi_{0, K}}^{\xi_{1, K}}
\frac{\partial^{j}J_2(0,\xi)}{\partial\eta^{j}} \hp_i'(\xi)
\hp_{k-1}(\xi) d\xi \nonumber \\
&\approx& \sum_{r = 1}^{\Nq}\left({\color{cyan}\xi_h^{-1}}
{\color{red}\frac{\partial^{j}J_0(0,\zeta_r)}{\partial\eta^{j}}}
{\color{blue}\cw_rp_i''(\czeta_r)p_{k-1}(\czeta_r)}
+ {\color{red}\frac{\partial^{j}J_2(0,\zeta_r)}{\partial\eta^{j}}}
{\color{blue}\cw_r p_i'(\czeta_r)p_{k-1}(\czeta_r)}\right).
\label{eq:vector_b_entries}
\end{eqnarray}
For computing the matrix $\mathbf{A}$, we note that
\begin{eqnarray*}
\frac{\partial^{j+2} \N_l(0, \xi)}{\partial \eta^{j+2}} &=&
\frac{1}{\eta_h^{j+2}} q_t^{(j+2)}(0)
p_{s}\left(\frac{\xi-\xi^{\mathtt{mid}}}{\xi_h}\right), \\
\frac{\partial^{j-l}}{\partial \eta^{j-l}}\left(\frac{\partial^2
\N_l(0, \xi)}{\partial \xi^2}\right) &=&
\frac{1}{\eta_h^{j-l}\xi_h^2}q_t^{(j-l)}(0)p_{s}''\left(\frac{\xi-\xi^{\mathtt{mid}}}{\xi_h}\right),
\\
\frac{\partial^{j-l+1} \N_l(0, \xi)}{\partial \eta^{j-l+1}} &=&
\frac{1}{\eta_h^{j-l+1}}
q_t^{(j-l+1)}(0)p_{s}\left(\frac{\xi-\xi^{\mathtt{mid}}}{\xi_h}\right), \\
\frac{\partial^{j-l}}{\partial \eta^{j-l}}\left(\frac{\partial
\N_l(0, \xi)}{\partial \xi}\right) &=&
\frac{1}{\eta_h^{j-l}\xi_h}q_t^{(j-l)}(0)
p_{s}'\left(\frac{\xi-\xi^{\mathtt{mid}}}{\xi_h}\right).
\end{eqnarray*}
Then, by \eqref{eq:eta_d_L}, we can compute the entries of the matrix
$A^{(j)}, ~0 \leq j \leq m-2$ by a numerical quadrature as follows
\begin{align}
&\left(A^{(j)}\right)_{k, l} = \int_{{\hat \Gamma}_{K_F}}\left(
\frac{\partial^j}{\partial \eta^j} \mathscr{L}(\N_l)(0,\xi)
\right)\hp_{k-1}(\xi) d\xi \nonumber \\
=& \int_{\xi_{0, K}}^{\xi_{1, K}} \left[\frac{\partial^{j+2}
\N_l(0, \xi)}{\partial \eta^{j+2}}
+ \sum_{i=0}^{j}\binom{j}{i}\left(
\frac{\partial^{i}J_0(0,\xi)}{\partial\eta^{i}}
\frac{\partial^{j-i}}{\partial
\eta^{j-i}}\left(\frac{\partial^2  \N_l(0, \xi)}{\partial
\xi^2}\right)  \right. \right. \nonumber \\
& \hspace{0.4in} \left. \left. +
\frac{\partial^{i}J_1(0,\xi)}{\partial\eta^{i}}
\frac{\partial^{j-i+1} \N_l(0, \xi)}{\partial \eta^{j-i+1}}  +
\frac{\partial^{i}J_2(0,\xi)}{\partial\eta^{i}}
\frac{\partial^{j-i}}{\partial \eta^{j-i}}\left(\frac{\partial
\N_l(0, \xi)}{\partial \xi}\right)\right) \right]\hp_{k-1}(\xi)
d\xi \nonumber \\
\begin{split}
\approx & \sum_{r = 1}^{\Nq} \left[{\color{cyan}\frac{\xi_h}{\eta_h^{j+2}}}
{\color{blue}\cw_rq_t^{(j+2)}(0) p_s(\czeta_r) p_{k-1}(\czeta_r)}
+ \sum_{i=0}^{j}\binom{j}{i}\left(
{\color{cyan}\frac{1}{\eta_h^{j-i}\xi_h}}
{\color{red}\frac{\partial^{i}J_0(0,\zeta_r)}{\partial\eta^{i}}}
{\color{blue}\cw_rq_t^{(j-i)}(0)p_{s}''(\czeta_r)
p_{k-1}(\czeta_r)} \right. \right. \\
&\left. \left. + {\color{cyan}\frac{\xi_h}{\eta_h^{j-i+1}}}
{\color{red}\frac{\partial^{i}J_1(0,\zeta_r)}{\partial\eta^{i}}}
{\color{blue}\cw_rq_t^{(j-i+1)}(0)p_{s}(\czeta_r)p_{k-1}(\czeta_r)}
+  {\color{cyan}\frac{1}{\eta_h^{j-i}}}
{\color{red}\frac{\partial^{i}J_2(0,\zeta_r)}{\partial\eta^{i}}}
{\color{blue}\cw_rq_t^{(j-i)}(0) p_{s}'(\czeta_r)p_{k-1}(\czeta_r)}
\right)  \right]. \label{eq:matrix_A_entries}
\end{split}
\end{align}
In formulas \eqref{eq:vector_b_entries} and
\eqref{eq:matrix_A_entries}, the quantities in blue  are in
terms of the polynomials $p_i$ and $q_i$ chosen in
\eqref{eq:pi-qi_polynomials}, as well as the nodes and weights of GLQ
on $[-1, 1]$. This means they are independent of the interface
elements or the interface itself, and they can be prepared before
implementing $\mathbf{b}^{(j)}(i)$ and $A^{(j)}$ over all interface
elements. The
quantities in red follow the formula given in
\eqref{eq:J_derivatives} which are determined by the Frenet apparatus
of the interface curve. Those quantities in cyan depend on the
interface element $K$, and they are computed according to
\eqref{eq:xi_01_etah_xih}.

We now discuss the implementation for the vector
$\mathbf{b}^{(j)}(i)$ and matrix $A^{(j)}$
defined in \eqref{eq:bji} and  \eqref{eq:Aj}, respectively. We start
from the preparation of
needed arrays. The sample \MATLAB program \verb!pipkValues.m! given in
\cite{adjeridMATLABImplementationGeometryConforming} is for preparing
the values of
$p_i^{(d)}(x)p_j(x), ~0 \leq i, j \leq m, ~d = 0, 1, 2$ at Gaussian
quadrature nodes in the reference interval $[-1, 1]$. It can be used as follows:
\begin{matlab}
pipk = pipkValues(degree, p_PolyFun, czeta)
\end{matlab}
For the inputs, \verb!degree! specifies the degree of the underlying
polynomial space $\mathbb{Q}_m$, i.e.,
the value of \verb!degree! is $m$, \verb!p_PolyFun(x, i, d)! is a
\MATLAB function for $p_i^{(d)}(x)$ used in
\eqref{eq:pi-qi_polynomials}, and \verb!czeta! is the array of
Gaussian quadrature nodes to be used. The output \verb!pipk! is a
4-dimensional array such that
\begin{eqnarray*}
\verb!pipk(:, :, d, r)! =
\Big(p_{i-1}^{(d-1)}(\czeta_r)p_{j-1}(\czeta_r)\Big)_{i, j = 1}^{m+1},
~1 \leq d \leq 3, ~1 \leq r \leq \Nq.
\end{eqnarray*}
The program \verb!qiValues.m! is for preparing $q_t^{(d)}(0), 0 \leq
t \leq m$, $0 \leq d \leq m$. This \MATLAB function can be used as follows:
\begin{matlab}
q0 = qiValues(degree, q_PolyFun)
\end{matlab}
The input argument \verb!degree! is as previously explained,
\verb!q_PolyFun(x, i, d)! is a \MATLAB function for evaluating
$q_i^{(d)}(x)$ used in \eqref{eq:pi-qi_polynomials}. The output
\verb!q0! is a 2-dimensional array:
\begin{eqnarray*}
\verb!q0! = \Big(q_{i-1}^{(d-1)}(0)\Big)_{i = 1, d = 1}^{m+1, m+1}.
\end{eqnarray*}
We also prepare the values for the $\eta$-partial derivatives
of $J_k(\eta, \xi), ~k = 0, 1, 2$ at $\eta =0$ before assembling the
matrices $A^{(j)}, 0 \leq j \leq m-2$. A sample \MATLAB function for
this task is \verb!J012Values.m! in
\cite{adjeridMATLABImplementationGeometryConforming}, and it can be
used as follows:
\begin{matlab}
J012 = J012Values(degree, CurveDG, zeta)
\end{matlab}
The input arguments \verb!degree, CurveDG! are explained before, and
\verb!zeta! is the array
of Gaussian quadrature nodes in $[\xi_{0K}, \xi_{1K}]$. The output
\verb!J012! is 3-dimensional array
such that
\begin{eqnarray*}
\verb!J012(:, k, :)! = \left(\frac{\partial^{i-1} J_k(0,
\zeta_j)}{\partial \eta^{i-1}}\right)_{i = 1, j=1}
^{m-1, \Nq}, \hspace{0.3in} k = 0, 1, 2.
\end{eqnarray*}

The sample \MATLAB function \verb!bijVector.m! in
\cite{adjeridMATLABImplementationGeometryConforming} is for
constructing the vector
$\mathbf{b}^{(j)}(i), ~0 \leq i \leq m, ~0 \leq j \leq m-2$. We can
use this function as follows:
\begin{matlab}
bij = bijVector(i, j, degree, J012, pipk, xi_h, cw)
\end{matlab}
Here, the inputs \verb!i! and \verb!j! are the integers { $i,j$} in
$\mathbf{b}^{(j)}(i)$,
the arguments \verb!degree, J012, pipk! are
as explained before, \verb!xi_h! is $\xi_h$
defined in \eqref{eq:xi_01_etah_xih}, and \verb!cw! is a vector
containing the Gaussian
quadrature weights on $[-1, 1]$. The output \verb!bij! is the vector
$\mathbf{b}^{(j)}(i)$ defined in \eqref{eq:bji}.

The \MATLAB function \verb!AjMatrix.m! in
\cite{adjeridMATLABImplementationGeometryConforming} is for
constructing the matrix
$A^{(j)}, ~0 \leq j \leq m-2$, which can be used as follows:
\begin{matlab}
Aj = AjMatrix(j, degree, J012, pipk, q0, eta_h, xi_h, cw)
\end{matlab}
Here, the argument \verb!j! is the integer $j$ in $A^{(j)}$, arguments
\verb!degree, J012, pipk, q0, xi_h, cw!
are as explained above, \verb!eta_h! is $\eta_h$ defined in
\eqref{eq:xi_01_etah_xih}. The output
\verb!Aj! is $A^{(j)}$ defined by \eqref{eq:Aj}.

Lastly, we note that the matrix $\mathbf{A}$ in \eqref{eq:Abc} might become
ill-conditioned for large degrees $m$ or small mesh sizes $h$. This
ill-conditioning stems from the scaling induced by the higher order
derivatives in $\p_{\eta^j}\L(\N_l)$ in formula \eqref{eq:Aj} since
$\N_l$ is obtained via a scaling of polynomial by $\eta_h$ and
$\xi_h$ which are comparable to the mesh size $h$. Hence, we expect
$\p_{\eta^j}\L(\N_l)\sim h^{-j-2}$ as $h\to 0^{+}$ which can be large
for high degrees $m$. To circumvent this issue, we proposed in
\cite{adjeridHighOrderGeometry2024} the use of the Jacobi
preconditionner $\mathbf{P}^{(1)}_{i,j}=\delta_{i,j}\mathbf{A}_{i,j}$, where
$\delta_{ij}$ is the Kronecker delta function.
We note, however, that other diagonal preconditioners
can be used such as the row normalizing preconditioner
$\mathbf{P}^{(2)}_{i,j}=\delta_{i,j}\norm{\mathbf{A}_{i,:}}$ where
the diagonal entries are the Euclidean norms of the rows of
$\mathbf{A}$. Therefore, $\mathbf{P}^{(2)}$ accounts for potentially
large entries in the matrix $\mathbf{A}$ that may not be on the diagonal.

To illustrate the importance of the preconditioning step, we consider two
examples where the interface is the unit circle parametrized
by $\xi \to (\cos\xi,\sin\xi)$. In the first example, we fix the
interface element
$K(h)=\frac{1}{\sqrt{2}}\left(1+\frac{h}{2}[-1,1]^2\right)$ with a constant
diameter $h=\frac{1}{2}$ and vary the degree $m=1,2,\ldots,8$. In the
second example, we fix the
degree $m=4$ and vary $h=2^{-1},2^{-2},\ldots,2^{-8}$. For both
examples, we present the condition number of the matrix $\mathbf{A}$
before and after applying the preconditionner $\mathbf{P}^{(1)}$ and
$\mathbf{P}^{(2)}$ in \autoref{sufig:cond_A_m} (Example 1) and
\autoref{subfig:cond_A_h} (Example 2). We observe that the condition
number of $\mathbf{A}$ increases exponentially as the polynomial
degree increases and as a power of $h^{-1}$ as $h\to 0$, and that, as
expected, the preconditioned matrices $\mathbf{P}^{(1)}\mathbf{A}$
and $\mathbf{P}^{(2)}\mathbf{A}$ are substantially better conditioned
than $\mathbf{A}$, especially for higher degrees $m$. We also note
that, for a fixed degree $m$, the preconditioned matrices are not
affected by the mesh size $h$ as observed in
\cite{adjeridHighOrderGeometry2024}.
\begin{figure}[ht]
\begin{subfigure}{.45\textwidth}
\centering
\includegraphics[scale=.8]{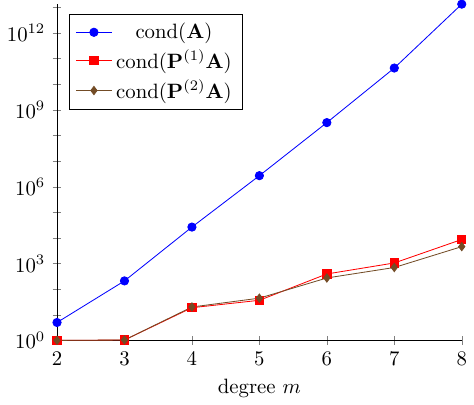}
\caption{Varying the degree with $h=2^{-1}$.}
\label{sufig:cond_A_m}
\end{subfigure}
\begin{subfigure}{.45\textwidth}

\centering
\includegraphics[scale=.8]{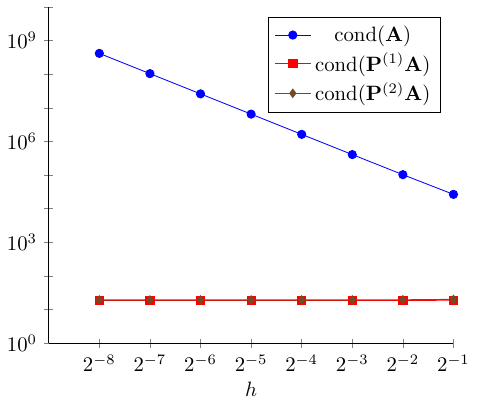}
\caption{Varying the size of the element with $m=4$.}
\label{subfig:cond_A_h}
\end{subfigure}
\caption{The condition number of $\mathbf{A}$ before and after the
preconditioning procedure under degree refinement (left, in
semi-log plot), and
under mesh refinement (right, in log-log plot)}
\label{fig:cond_A}
\end{figure}

\subsection{A generalized initial construction of a GC-IFE basis}
\label{sec:method_general}

The construction procedure discussed in the previous subsection is
based on the idea developed in
\cite{adjeridHighOrderGeometry2024,meghaichiHigherOrderImmersed2024}.
A notable feature of this construction procedure is that $m(m+1)$
functions among the $(m+1)^2$ basis functions are explicitly
described in simple formulas
in \eqref{eqn:hphi_Kp_Km_ij>1}. Computations are required only for
the coefficients $\mathbf{c}^{(i)}$
of the remaining $m+1$ basis functions $\hphi_{i,0}(\eta, \xi), ~0
\leq i \leq m$. Nevertheless, it is well known that a basis for a
polynomial space in terms of simple polynomials can lead to
ill-conditioned computational procedures. For example, when using the
simple and elegant monomial basis $\{x^i\}_{i = 0}^m$ to compute the
$L^2$ projection $p(x) = \sum_{i=0}^m c_i x^i$ of a function $f(x)$ onto
the polynomial space $\mathbb{P}_m$, the matrix $\A$ in the linear
system used to determine the coefficients
$c_i, 0 \leq i \leq m$ is a Hilbert matrix, which is well known for
its ill-conditioning. In contrast, if
the Legendre polynomials are used as a basis of $\mathbb{P}_m$ to
find the $L^2$ projection of $f(x)$, the matrix $\A$ is optimally
conditioned, i.e., it has the best possible condition number. This
consideration motivates us to seek bases for the GC-IFE space
$\V^m_{\beta}(K)$ in more general form than the one described in the
previous section.
Moreover, this approach allows us to consider a more general form of interface
conditions such as the ones used for hyperbolic interface problems,
where the extended
interface condition \eqref{eqn:IFE_extended} may be different, or in situations
where the diffusion coefficient $\beta$ is not piecewise constant.

We start from the basis $\R=\{\R_l\}_{l=1}^{(m+1)^2}$ for the
polynomial space $\mathbb{Q}_m$ where
\begin{equation}
\R_{(m+1)t+s+1}(\eta,\xi) = \hq_t(\eta)\hp_s(\xi) = q_t
\left(\frac{\eta}{\eta_h}\right)p_s
\left(\frac{\xi-\ximid}{\xi_h}\right),\qquad 0\le s, t \le m,
\label{eqn:R_def}
\end{equation}
with the univariate polynomials $p_i$ and $q_i$ as given  in
\eqref{eq:pi-qi_polynomials} and \eqref{eq:q_poly}. Then, in contrast
to the basis in the special format described in
\eqref{eqn:hphi_Kp_Km_ij>1} and \eqref{eqn:phi_i0}, we proceed to
construct a basis $\hB = \{\hlambda_j\}_{j=1}^{(m+1)^2}$ in a more
general format for the IFE space $\hV^m_{\beta}(\hK_F)$ on the Frenet
fictitious element $\hK_F$ of an interface element $K$, such that
the functions in this basis have the following form:
\begin{equation}
\hlambda_j(\eta,\xi)=
\begin{cases}
\hlambda_j^-(\eta,\xi) = \ds \sum_{i=1}^{(m+1)^2}
C^-_{i,j}\R_i(\eta,\xi), & \eta<0, \\[12pt]
\hlambda_j^+(\eta,\xi) =\ds \sum_{i=1}^{(m+1)^2}
C^+_{i,j}\R_i(\eta,\xi), & \eta>0,
\end{cases} \hspace{0.2in} 1 \leq j \leq (m+1)^2 \label{eq:basisFun_lambda_j},
\end{equation}
in which the coefficients $C^\pm_{i,j}, ~1 \leq i \leq m+1)^2$, are to
be determined such that ${\hat u}(\eta, \xi) = \hlambda_j(\eta,\xi)$
satisfies the transformed interface jump conditions specified in
\eqref{eqns:IFE_conditions} along the interface line segment $\hGamma_{K_F}$.

Specifically, imposing the interface jump conditions on
$\hlambda_j(\eta,\xi)$ leads to the following linear equations for
the coefficients $C^\pm_{i,j},~1 \leq i \leq (m+1)^2$:
\begin{eqnarray*}
&&\sum_{i=1}^{(m+1)^2} \left(\int_{\xi_{0, K}}^{\xi_{1, K}} \R_i(0,
\xi) \hp_{k-1}(\xi) d\xi\right)C^+_{i, j}
= \sum_{i=1}^{(m+1)^2} \left(\int_{\xi_{0, K}}^{\xi_{1, K}} \R_i(0,
\xi) \hp_{k-1}(\xi) d\xi\right)C^-_{i, j} \\
&&\beta^+\sum_{i=1}^{(m+1)^2} \left(\int_{\xi_{0, K}}^{\xi_{1, K}}
\Big(\partial_\eta \R_i(0, \xi)\Big) \hp_{k-1}(\xi) d\xi\right)C^+_{i, j}
= \beta^-\sum_{i=1}^{(m+1)^2} \left(\int_{\xi_{0, K}}^{\xi_{1, K}}
\Big(\partial_\eta \R_i(0, \xi)\Big) \hp_{k-1}(\xi) d\xi\right)C^-_{i, j} \\
&&\beta^+\sum_{i=1}^{(m+1)^2} \left(\int_{\xi_{0, K}}^{\xi_{1, K}}
\Big(\frac{\p^n}{\p\eta^n}\L\R_i(0, \xi)\Big) \hp_{k-1}(\xi)
d\xi\right)C^+_{i, j}
= \beta^-\sum_{i=1}^{(m+1)^2} \left(\int_{\xi_{0, K}}^{\xi_{1, K}}
\Big(\frac{\p^n}{\p\eta^n}\L\R_i(0, \xi)\Big) \hp_{k-1}(\xi)
d\xi\right)C^-_{i, j}, \\[10pt]
&& \text{for}~~k = 1, 2, \ldots, m+1, ~n = 0, 1, \ldots, m-2.
\end{eqnarray*}
Letting $C_j^\pm$ be the column vector formed by the coefficients
$C^\pm_{i,j},~1 \leq i \leq (m+1)^2$,
we can write this system of linear equations in a matrix-vector form as follows:
\begin{eqnarray}
\tilde{\mathbf{A}} C_j^+ = J\tilde{\mathbf{A}}C_j^-
~~\text{with}~~J=\operatorname{diag}\left(\underbrace{1,1,\dots,1}_{m+1},
\frac{\beta^-}{\beta^+},\frac{\beta^-}{\beta^+},\dots,\frac{\beta^-}{\beta^+}\right),
~~\tilde{\mathbf{A}} =
\begin{bmatrix}
B^{(0)} \\
B^{(1)} \\
{\tilde A}^{(0)} \\
\vdots \\
{\tilde A}^{(m-2)}
\end{bmatrix}, \label{eq:vector_coeff_eq}
\end{eqnarray}
where the sub-matrices of $\tilde{\mathbf{A}}$ are
\begin{eqnarray}
&&B^{(0)} = \left(\int_{\xi_{0, K}}^{\xi_{1, K}} \R_i(0, \xi)
\hp_{k-1}(\xi) d\xi\right)_{k, i =1}^{m+1, (m+1)^2},
~~B^{(1)} =  \left(\int_{\xi_{0, K}}^{\xi_{1, K}} \Big(\partial_\eta
\R_i(0, \xi)\Big) \hp_{k-1}(\xi) d\xi\right)_{k, i = 1}^{m+1,
(m+1)^2},  \label{eq:B_matrices}\\[10pt]
&&{\tilde A}^{(n)} = \left(\int_{\xi_{0, K}}^{\xi_{1, K}}
\Big(\frac{\p^n}{\p\eta^n}\L\R_i(0, \xi)\Big) \hp_{k-1}(\xi)
d\xi\right)_{k, i = 1}^{m+1, (m+1)^2}, ~~0 \leq n \leq m-2.
\label{eq:tA_matrices}
\end{eqnarray}
By Lemma 1 in \cite{adjeridHighOrderGeometry2024}, for any choice of
$\hlambda_j^-(\eta,\xi) \in \mathbb{Q}_m$, a polynomial
$\hlambda_j^+(\eta,\xi) \in \mathbb{Q}_m$ is uniquely determined by
the interface jump conditions \eqref{eqns:IFE_conditions}. This means
that $C_j^+$ is uniquely determined by $\tilde{\mathbf{A}} C_j^+ =
J\tilde{\mathbf{A}}C_j^-$ for any choice of $C_j^-$; consequently,
the matrix $\tilde{\mathbf{A}}$ is nonsingular.
This suggests that, following the extension idea in the IFE literature
\cite{adjeridFrenetImmersedFinite2025,
guoDesignAnalysisApplication2019, guoHigherDegreeImmersed2019}, we
can determine an IFE function $\hlambda_j(\eta,\xi)$ by a choice of
its component $\hlambda_j^-(\eta,\xi)$. Hence,
choosing a basis $\{\hlambda_j^-(\eta,\xi)\}_{j = 1}^{(m+1)^2}$ for
the polynomial space $\mathbb{Q}_m$, this extension procedure extends
$\{\hlambda_j^-(\eta,\xi)\}_{j = 1}^{(m+1)^2}$ to
$\{\hlambda_j^+(\eta,\xi)\}_{j = 1}^{(m+1)^2}$ which yields a basis
$\{\hlambda_j\}_{j=1}^{(m+1)^2}$ for the IFE space $\hV^m_{\beta}(\hK_F)$.
Since the matrix $J$ is nonsingular, this extension procedure can be
reversed, i.e., we can first choose a
basis $\{\hlambda_j^+(\eta,\xi)\}_{j = 1}^{(m+1)^2}$ for
$\mathbb{Q}_m$ and then extend them to form a basis
$\{\hlambda_j\}_{j=1}^{(m+1)^2}$ for the IFE space $\hV^m_{\beta}(\hK_F)$.

Computationally, this extension procedure can be carried out
efficiently in a collective way as follows. Let
$C^\pm = [C_1^\pm, C_2^\pm, \ldots, C_{(m+1)^2}^\pm]$, then, by
\eqref{eq:vector_coeff_eq}, matrices
$C^-$ and $C^+$ satisfy the following equation:
\begin{eqnarray}
\tilde{\mathbf{A}} C^+ = J\tilde{\mathbf{A}}C^-. \label{eq:matrix_coeff_eq}
\end{eqnarray}
Hence, the extension can be carried out by choosing a nonsingular
matrix $C^\pm$, which consequently
guarantees $\{\hlambda_j^\pm(\eta,\xi)\}_{j = 1}^{(m+1)^2}$ is a
basis of $\mathbb{Q}_m$, and computing
$C^\mp$ from \eqref{eq:matrix_coeff_eq}. Then use the $j$-th column
of $C^\pm$ as the coefficients
$C_{i,j}^\pm, ~1 \leq i \leq (m + 1)^2$ in
\eqref{eq:basisFun_lambda_j} to form the
$j$-th basis function $\hlambda_j(\eta,\xi)$. As for the choice of
$C^{-}$ or $C^{+}$, if $\abs{K^+}$ is larger than $\abs{K^-}$, we
choose $C^+$ to be the identity. However, this choice is not critical
because of the recommended following reconstruction to be discussed in
the next section.

It is also interesting to note that the construction procedure
presented in \autoref{sec:method_1} is a special case of this
general construction procedure based on extension. This is because
the  coefficient vectors $\c^{(0)},\c^{(1))},\ldots,\c^{(m)}$
are determined
by \eqref{eq:Abc} in the initial construction procedure in
\autoref{sec:method_1}. Then, it can be verified that
the following two matrices satisfy \eqref{eq:matrix_coeff_eq} :
\begin{equation}
C^- =
\begin{bmatrix}\ds
I_{m+1} & \mathbf{0}_{(m+1)\times m(m+1)} \\
\mathbf{0}_{m(m+1)\times (m+1)}  & \ds \frac{1}{\beta^-}I_{m(m+1)}
\end{bmatrix}, ~~
C^+=
\begin{bmatrix}
I_{m+1} & \mathbf{0}_{(m+1)\times (m+1)}& \mathbf{0}_{(m+1)\times (m^2-1)}  \\
\mathbf{0}_{(m+1)\times (m+1)} &  \ds \frac{1}{\beta^+}I_{m+1}&
\mathbf{0}_{(m+1)\times (m^2-1)} \\
\c & \mathbf{0}_{(m^2-1)\times (m+1)}  & \ds \frac{1}{\beta^+}I_{m^2-1}&
\end{bmatrix},
\label{eqn:C0p_C0m}
\end{equation}
where $\c=[\c^{(0)},\c^{(1))},\dots,\c^{(m)}]$. In this way, since
these two matrices depend on $\c$ only, the
initial construction in Subsection \autoref{sec:method_1} has been
interpreted as a general initial construction.

Similarly to the matrix $\mathbf{A}$ in the previous section, the
matrix $\tilde{\mathbf{A}}$ needs to be rescaled in order to improve
its condition number. Again, we propose to use either the Jacobi
preconditioner
$\tilde{\mathbf{P}}^{(1)}_{i,j}=\delta_{i,j}\tilde{\mathbf{A}}_{i,j}$
or the row-normalizing preconditioner
$\tilde{\mathbf{P}}^{(2)}_{i,j}=\delta_{i,j}\norm{\tilde{\mathbf{A}}_{i,:}}$.
To illustrate the effectiveness of these two preconditioners, we
present the condition number of $\tilde{\mathbf{A}}$ before and after
applying one of the preconditioners $\tilde{\mathbf{P}}^{(k)}$ for
$k=1,2$, using the configurations described in the previous section.
We observe in \autoref{fig:cond_At} that the preconditioners improve
the condition number of
$\tilde{\mathbf{A}}$ substantially for higher degrees and keep the
condition number stable as the mesh size $h\to 0^+$.
\begin{figure}[ht]
\begin{subfigure}{.45\textwidth}
\centering
\includegraphics[scale=.8]{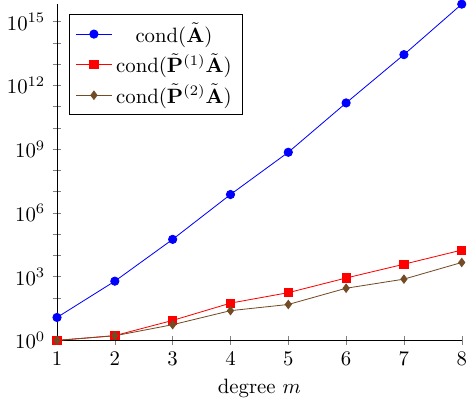}
\caption{Varying the degree with $h=2^{-1}$.}
\end{subfigure}
\begin{subfigure}{.45\textwidth}

\centering
\includegraphics[scale=.8]{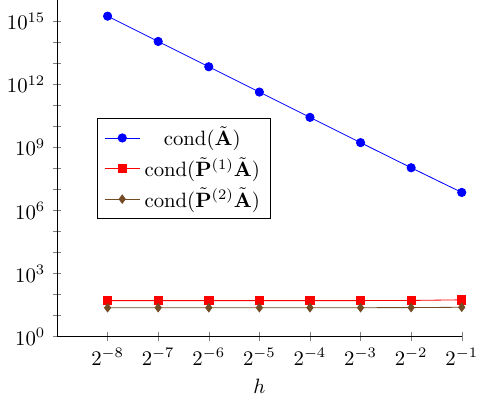}
\caption{Varying the size of the element with $m=4$}

\end{subfigure}
\caption{The condition number of $\tilde{\mathbf{A}}$ before and
after the preconditioning procedure. }
\label{fig:cond_At}
\end{figure}

{ A sample \MATLAB program of this procedure that constructs  the basis
$\{\lambda_i\}_i=\{\hat{\lambda}_i\circ \RG\}_{i}$
is  \verb!GCIFEbasis.m!  made available in
\cite{adjeridMATLABImplementationGeometryConforming}
and can be used as follows:}
\begin{matlab}
phi = GCIFEbasis(x, y, d1, d2, degree, IntElemInfo, CurveDG, pPolyFun, qPolyFun, cm, cp)
\end{matlab}
The input arguments \verb!x, y! are two vectors of the
same length specifying the $x$-$y$ coordinates for evaluating
$\lambda_{\ell}(x, y)$ or its first-order partial derivative,
\verb!d1, d1! are two integers specifying which partial derivative is evaluated,
\verb!CurveDG! is as explained before,
\verb!IntElemInfo! is a structure whose fields
\verb!eta_h, xi_h, xi_mid! are for $\eta_h, \xi_{h}, \ximid$,
\verb!pPolyFun, qPolyFun! are functions for the chosen $p_i(x),
q_i(x)$ polynomials in \eqref{eq:pi-qi_polynomials},
\verb!degree! specifies the degree associated with the underlying polynomial space
$\mathbb{Q}_m$,
\verb!c_m! and \verb!c_p! are columns of $C^-$ and $C^+$, respectively,
corresponding to the same basis function. That is,
$\mathtt{c\_m}=C^{-}_{\ell}$,
and $\mathtt{c\_p}=C^{+}_{\ell}$ for some index $1\le \ell\le (m+1)^2$.
When
\verb!d1 = 0, d2 = 0!, \verb!phi! {is the value of $\lambda_{\ell}(x, y)$,}
when \verb!d1 = 1, d2 = 0!, \verb!phi! is the value the $x$-derivative of
$\lambda_{\ell}$ at $(x, y)$,
when \verb!d1 = 0, d2 = 1!, \verb!phi! is the value of the $y$-derivative of
$\lambda_{\ell}$ at $(x, y)$, and
when \verb!d1 = 1, d2 = 1!, \verb!phi! is the gradient of $\lambda_{\ell}$ at
$(x,y)$. We note  that we included
\verb!GCIFEbasis! here for
educational purposes only to allow the reader to experiment with
individual basis
functions. In practice, it is preferable to assemble the local matrices using a
Generalized Vandermonde matrix.

Next, we discuss computational details for this general construction
procedure. Let
$s, t$ be such that $i = (m+1)t+s+1$. For the entries of matrix
$B^{(0)}$ in \eqref{eq:B_matrices},
by \eqref{eqn:R_def}, we have
\begin{eqnarray*}
B_{k, i}^{(0)} = \int_{\xi_{0, K}}^{\xi_{1, K}} \R_i(0, \xi)
\hp_{k-1}(\xi) d\xi &=& \int_{\xi_{0, K}}^{\xi_{1, K}}
\big(q_t(0) \hp_s(\xi)\big) \hp_{k-1}(\xi) d\xi \approx q_t(0)\sum_{r
= 1}^{\Nq} \xi_h\cw_r p_s(\czeta_r) p_{k-1}(\czeta_r)
\end{eqnarray*}
where, as before, $\czeta_r, \cw_r, ~1 \leq r \leq \Nq$ are the
Gauss-Legendre quadrature (GLQ) nodes and weights on $[-1,1]$. For
the entries of
$B^{(1)}$, we have
\begin{eqnarray*}
B_{k, i}^{(1)} = \int_{\xi_{0, K}}^{\xi_{1, K}} \Big(\partial_\eta
\R_i(0, \xi)\Big) \hp_{k-1}(\xi) d\xi &=&
\int_{\xi_{0, K}}^{\xi_{1, K}}
\Big(\frac{1}{\eta_h}q_t'(0)\hp_s(\xi)\Big) \hp_{k-1}(\xi) d\xi
\approx \frac{q_t'(0)}{\eta_h}\sum_{r = 1}^{\Nq} \xi_h\cw_r
p_s(\czeta_r) p_{k-1}(\czeta_r)
\end{eqnarray*}
Following the same derivation as \eqref{eq:matrix_A_entries}, we can compute the
entries of ${\tilde A}^{(n)}$ as follows:
\begin{align*}
&{\tilde A}_{k, i}^{(n)} = \int_{\xi_{0, K}}^{\xi_{1, K}}
\Big(\frac{\p^n}{\p\eta^n}\L\R_i(0, \xi)\Big) \hp_{k-1}(\xi) d\xi \\
=& \int_{\xi_{0, K}}^{\xi_{1, K}} \left[\frac{\partial^{n+2} \R_i(0,
\xi)}{\partial \eta^{n+2}}
+ \sum_{l=0}^{n}\binom{n}{l}\left(
\frac{\partial^{l}J_0(0,\xi)}{\partial\eta^{l}}
\frac{\partial^{n-l}}{\partial \eta^{n-l}}\left(\frac{\partial^2
\R_i(0, \xi)}{\partial \xi^2}\right)  \right. \right. \\
& \hspace{0.4in} \left. \left. +
\frac{\partial^{l}J_1(0,\xi)}{\partial\eta^{l}}
\frac{\partial^{n-l+1} \R_i(0, \xi)}{\partial \eta^{n-l+1}}  +
\frac{\partial^{l}J_2(0,\xi)}{\partial\eta^{l}}
\frac{\partial^{n-l}}{\partial \eta^{n-l}}\left(\frac{\partial
\R_i(0, \xi)}{\partial \xi}\right)\right) \right]\hp_{k-1}(\xi) d\xi
\nonumber \\
\begin{split}
\approx& \sum_{r = 1}^{\Nq} \xi_h\cw_r \left[\frac{1}{\eta_h^{n+2}}
q_t^{(n+2)}(0) p_s(\czeta_r)
+ \sum_{l=0}^{n}\binom{n}{l}\left(
\frac{\partial^{l}J_0(0,\zeta_r)}{\partial\eta^{l}}
\frac{1}{\eta_h^{n-l}\xi_h^2}q_t^{(n-l)}(0)p_{s}''(\czeta_r)
\right. \right. \\
&\left. \left. +
\frac{\partial^{l}J_1(0,\zeta_r)}{\partial\eta^{l}}
\frac{1}{\eta_h^{n-l+1}} q_t^{(n-l+1)}(0)p_{s}(\czeta_r) +
\frac{\partial^{l}J_2(0,\zeta_r)}{\partial\eta^{l}}
\frac{1}{\eta_h^{n-l}\xi_h}q_t^{(n-l)}(0) p_{s}'(\czeta_r)
\right)  \right]p_{k-1}(\czeta_r).
\end{split} \nonumber
\end{align*}

We note that the procedure for constructing the
matrix ${\tilde A}^{(n)}$ is similar to that for the matrix $A^{(n)}$
given in \eqref{eq:Aj} because
of the similarity between $\R_i(\eta, \xi)$ and $\N_l(\eta, \xi)$.
However, ${\tilde A}^{(n)}$ has $2m + 2$ more columns and the index
$i$ of $\R_i(\eta, \xi)$ is given by the formula $i = (m+1)t+s+1$
which is different from the formula $l = (t-2)(m+1) + (s+1)$ given in
\eqref{eqn:lexicographical} for the index of
$\N_l(\eta, \xi)$. \\

Matrices $B^{(0)}, B^{(1)}$ and ${\tilde A}^{(n)}, ~0 \leq n \leq
m-2,$ can be assembled efficiently in a way very similar to the
assembly of the vector $\mathbf{b}^{(j)}(i)$ and the matrix $A^{(j)}$
discussed in \autoref{sec:method_1} with the data arrays
\verb!pipk, q0, J012! prepared by the programs provided there.
The \MATLAB function \verb!AtnMatrix.m! in
\cite{adjeridMATLABImplementationGeometryConforming} is a sample
program to assemble the matrix
${\tilde \A}$ in \eqref{eq:vector_coeff_eq} (or
\eqref{eq:matrix_coeff_eq}). It can be used as follows:
\begin{matlab}
At = AtMatrix(degree, J012, pipk, q0, eta_h, xi_h, cw)
\end{matlab}
All the input arguments are the same as described in
\autoref{sec:method_1}, and the output
argument \verb!At! is the matrix ${\tilde A}$. 

The GC-IFE basis functions in the general form
\eqref{eq:basisFun_lambda_j} also streamline the assembly of the
related local matrices arising in the implementation of GC-IFE
methods for solving interface problems. To convey the essential ideas
concisely, we will discuss the assembly of the local mass matrix on
an interface element $K$, with the understanding that analogous
procedures apply to the construction of other types of local matrices.

With a basis $\hB = \{\hlambda_j\}_{j=1}^{(m+1)^2}$ for the IFE space
$\hV^m_{\beta}(\hK_F)$ on the Frenet fictitious element $\hK_F$,
according to \cite{adjeridHighOrderGeometry2024}, we know that
$\{\lambda_j = \hlambda_j \circ R_\Gamma\}_{j=1}^{(m+1)^2}$ is
a basis for the local GC-IFE space $\V^m_{\beta}(K)$ on an interface
element $K$. Then, the local mass matrix $M = (M_{i,j})_{i, j
= 1}^{(m+1)^2}$ associated with this basis of $\V^m_{\beta}(K)$ is
\begin{equation}
M_{i,j}=\int_K \lambda_i\lambda_j dxdy
= \int_K (\hat{\lambda}_i\circ
\RG)(\hat{\lambda}_j\circ\RG)dxdy,\qquad 1\le i,j\le (m+1)^2.
\label{eq:mass_Mij}
\end{equation}
First, because the GC-IFE functions are defined piecewise according
to the interface, the integration for
$M_{i,j}$ should be carried out as the summation of the integrations
on two sub-elements $K^s = K \cap \Omega^s,~s = -, +$. Second,
numerical quadrature is inevitable for preparing the entries of the
mass matrix because, in general, the underlying basis functions
$\lambda_j, ~1 \leq j \leq (m+1)^2$ are non-polynomials and
the sub-elements $K^s, s = -, +$ are not polygons.

Let
$\{(\x_k^{\pm},w_k^{\pm})\}_{k=1}^{\mathtt{N}_{\mathtt{q}}^{\pm}}$
be the nodes and weights for the chosen numerical quadrature rules on
$K^{\pm}$ and let $\hx^{\pm}_k =\RG(\x^{\pm}_k), ~1 \leq k \leq
\mathtt{N}_{\mathtt{q}}^\pm$. Then, we have the following formula for
computing the entries of the mass matrix:
\begin{eqnarray*}
M_{i,j} &=& \int_K \lambda_i\lambda_jdxdy = \int_K
(\hat{\lambda}_i\circ \RG)(\hat{\lambda}_j\circ\RG)dxdy \\
&=& \int_{K^-} (\hat{\lambda}_i\circ \RG)(\hat{\lambda}_j\circ\RG)
dxdy+ \int_{K^+} (\hat{\lambda}_i\circ \RG)(\hat{\lambda}_j\circ\RG)dxdy \\
&\approx& {\sum_{s \in \{-,
+\}}\sum_{k=1}^{\mathtt{N}_{\mathtt{q}}^{s}}
w_k^{s}\big(\hlambda_i\circ R_\Gamma(\x_k^s)\big)
\big(\hlambda_j\circ R_\Gamma(\x_k^s)\big) = \sum_{s \in \{-, +\}}
\sum_{k=1}^{\mathtt{N}_{\mathtt{q}}^{s}}
w_k^{s}}\hlambda_i(\hx_k^s)\hlambda_j(\hx_k^s).
\end{eqnarray*}
Furthermore, we let
{
\begin{eqnarray*}
V^\pm = \Big(V_{k,i}^\pm  \Big)_{k,i= 1}^{
\mathtt{N}_{\mathtt{q}}^\pm,(m+1)^2}  =
\Big(\hat{\lambda}_i(\hx_k^{\pm})\Big)_{k,i= 1}^{
\mathtt{N}_{\mathtt{q}}^\pm,(m+1)^2}, ~~~~
L^\pm = \Big(L_{k,i}^\pm  \Big)_{k,i= 1}^{
\mathtt{N}_{\mathtt{q}}^\pm,(m+1)^2} =
\Big(\R_i(\hx_k^{\pm})\Big)_{k,i= 1}^{ \mathtt{N}_{\mathtt{q}}^\pm, (m+1)^2},
\end{eqnarray*}
}
be the generalized Vandermonde {matrices} associated with the basis
$\{\hat{\lambda}_i\}_{i=1}^{(m+1)^2}$
of $\V^m_{\beta}(K)$ and the basis $\{\R_i\}_{i = 1}^{(m+1)^2}$ of
$\mathbb{Q}_m$, respectively. Then,
we have $V^\pm = L^\pm C^{\pm}$ and
\begin{align}
\begin{split}
M = \big(M_{i,j}\big)_{i, j = 1}^{(m+1)^2} \approx& \left(\sum_{s
\in \{-, +\}} \sum_{k=1}^{\mathtt{N}_{\mathtt{q}}^{s}}
w_k^{ s}\hlambda_i(\hx_k^s)\hlambda_j(\hx_k^s)\right)_{i, j =
1}^{(m+1)^2} = \sum_{s \in \{-, +\}} \left(V^{ s } \right)^T W^{ s}
V^{ s} \\[5pt]
=&\left(L^-C^-\right)^T W^-L^- C^- +
\left(L^+C^+\right)^T W^+L^+ C^+
= M_q, \\[5pt]
\text{where} \hspace{0.1in} W^\pm =&
W^{\pm}=\operatorname{diag}(w_1^{\pm},w_2^{\pm},\dots,w_{\mathtt{N}_{\mathtt{q}}^{\pm}}^\pm),
\end{split}\label{eq:MassMatrix_matrixform}
\end{align}
where the subscript $q$ in $M_q$ emphasizes that it is an
approximation to the mass matrix $M$ with a chosen quadrature rule.
We note that matrix $L$ can be easily formed by evaluating
polynomials (not piecewise polynomials) $\R_j(\hx)$ defined by
\eqref{eqn:R_def} at the chosen quadrature nodes. This implies that
once an GC-IFE basis has been constructed in the format defined by
\eqref{eq:basisFun_lambda_j}, i.e., once the coefficient matrices
$C^\pm$ have been determined, the assembly of the mass matrix $M$ can
be efficiently accomplished by straightforward matrix multiplication
according to \eqref{eq:MassMatrix_matrixform}.

To assemble the load vector for a given function $f$, we follow a similar
procedure using the quadrature rules described above. More precisely, let
$\mathbf{f}=\left( \int_K f \lambda_i dxdy  \right)_{i=1}^{(m+1)^2}$. Then,

\begin{equation}
\mathbf{f}= V^-W^-\mathbf{r}^- + V^+W^+\mathbf{r}^+,
\qquad \text{where }\
\mathbf{r}^{\pm}=(f(\x_1^{\pm}),f(\x_2^{\pm})
,\dots,f(\x_{\mathtt{N}_{\mathtt{q}}^\pm}^{\pm}))^T.
\end{equation}
Hence, once the quadrature rule is obtained, and the Vandermonde
matrices $V^{\pm}$ are computed, we can assemble
the local mass matrix and the local load vector directly using matrix-matrix or
matrix-vector multiplications.

A classical way to choose a quadrature rule on
the sub-elements $K^s = K \cap \Omega^s, s = -, +$ of an interface
element $K$ is to partition
them into triangles according to the interface $\Gamma$ such that
each of these triangles has at most one
curved edge, see the illustration in
\autoref{subfig:traditional_quad_triangles} and explanations about such
partition  in \cite[Sec.
4.1.4]{moonImmersedDiscontinuousGalerkin2016}. These triangles are
homeomorphic images of the reference triangle with vertices $(0, 0),
(1, 0), (0, 1)$.
The homeomorphism is either the usual affine mapping if such a
triangle in $K^\pm$ is
formed with straight edges; otherwise, the homeomorphism can be
constructed using the
parametrization $\g(\xi)$ of the interface curve $\Gamma$ through well-known
procedures such as the one presented in
\cite{zlamalCurvedElementsFinite1973}. We then transform the nodes
and weights of a chosen quadrature on the reference triangle to
triangles in $K^\pm$ to
form the quadrature nodes and weight
$\{(\x_k^{\pm},w_k^{\pm})\}_{k=1}^{\mathtt{N}_{\mathtt{q}}^{\pm}}$ on
the sub-elements $K^\pm$. In essence, this constitutes a simple or
a composite quadrature rule on $K^-$ and $K^+$, respectively. When the
sub-elements $K^\pm$ are two curved quadrilaterals, as illustrated in
\autoref{subfig:traditional_quad_rectangles},
we can use the parametrization $\g(\xi)$ of the interface curve
to construct homeomorphic mappings
\cite{koprivaImplementingSpectralMethods2009} between $K^\pm$ and the
reference rectangle with vertices $(0, 0), (1, 0), (1, 1), (0, 1)$.
Then quadrature nodes and weight on $K^\pm$ are formed by
transforming their counterparts in a chosen quadrature rule on the
reference rectangle. On the reference rectangle, we use the
tensor product of the one dimensional Gauss-Legendre quadrature rule, and on
the reference triangle, we use the  Stroud conical product rule
\cite{stroudApproximateCalculationMultiple1971}. Although, other quadrature
rules on the unit triangle may be used such as the ones described in
\cite{lynessModerateDegreeSymmetric1975} and
\cite{dunavantHighDegreeEfficient1985}, we chose to use Stroud's rule since it
can be formed directly from one dimensional rules, which can be calculated and
efficiently using the well known Golub-Welsh algorithm
\cite{golubCalculationGaussQuadrature1969}
or Newton's iteration with a suitable initial guess
\cite{haleFastAccurateComputation2013}. We note that
modern approaches, such as Saye's algorithm
\cite{sayeHighorderQuadratureMethods2015}, can also be employed to
determine the quadrature nodes and weights on the sub-elements
$K^\pm$ according to the interface $\Gamma$.

The \MATLAB function \verb!TraditionalQuadrature.m! in
\cite{adjeridMATLABImplementationGeometryConforming} is
a sample program for generating quadrature nodes and weight on an
interface element $K$, and it can be used as follows:
\begin{matlab}
[xy_m, w_m, xy_p, w_p] = TraditionalQuadrature(elem, CurveDG, xiGuess, n_qp)
\end{matlab}
The input arguments \verb!CurveDG! and \verb!xiGuess! are as
described above, \verb!elem! is a $2 \times 4$ array whose $i$-th
column provides the $x$-$y$ coordinates for the $i$-th vertex of the
interface element $K$, \verb!n_qp! is a positive integer for
determining the number of quadrature points in {each direction} on
each subelement for a
chosen quadrature rule. The output arguments \verb!xy_m, w_m! are
quadrature nodes and weights for $K^- = K \cap \Omega^-$, and
\verb!xy_p, w_p! are quadrature nodes and weights for $K^+ = K \cap
\Omega^+$. We have assumed that the normal vector ${\bf n}(t)$
defined by the parametrization of the interface curve $\Gamma$ points
to $\Omega^+$. The plots in \autoref{fig:traditional_quad} provide
illustrations of quadrature nodes generated inside interface elements.

\begin{figure}[htbp]
\begin{center}
\begin{subfigure}{.335\textwidth}
\includegraphics[width=\textwidth]{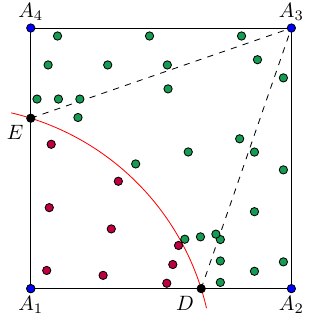}
\caption{Type I element.}
\label{subfig:traditional_quad_triangles}
\end{subfigure}\hspace{0.3in}
\begin{subfigure}{.32\textwidth}
\includegraphics[width=\textwidth]{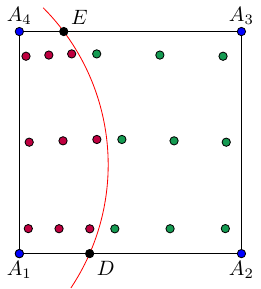}
\caption{Type II element.}
\label{subfig:traditional_quad_rectangles}
\end{subfigure}
\end{center}
\caption{A quadrature rule constructed on an interface element of type I (left)
and type II (right) with \texttt{n\_qp=3}.}
\label{fig:traditional_quad}
\end{figure}

\subsection{Reconstruction of a GC-IFE basis}
\label{subsec:Reconstruction}

The construction procedures for the GC-IFE bases discussed thus far,
including those in \autoref{sec:method_1} and
\autoref{sec:method_general}, share a significant limitation: they focus
solely on enforcing the jump conditions. While this approach aligns
with the primary goal of solving interface problems, it critically
overlooks a crucial aspect of computational performance. These
procedures don't account for how the constructed GC-IFE basis will
actually behave when deployed to generate approximations for the
interface problem. This oversight directly motivates us to introduce
the reconstruction designed specifically to enhance the GC-IFE basis
by incorporating this critical computational aspect, aiming for
improved performance in solving interface problems.

Our approach is to make sure that the mass matrix of the GC-IFE basis
is well conditioned because of the following important factors: (1) The
mass matrix is directly related to the $L^2$ projection onto the
constructed GC-IFE space which gauges the approximation capability of
the GC-IFE functions; (2)
The mass matrix plays an important role in schemes for solving interface
problems, especially those time-dependent interface problems. In
other words, a well-conditioned mass matrix is closely associated
with computational stability when the constructed GC-IFE basis is
used to produce approximations for interface problems. These
considerations are also inspired by prior work in the literature.
In the traditional DG methods, orthogonal
bases are constructed on a reference element, then they are mapped to
the physical element, such as tensor-product Legendre polynomials on
rectangles and hexahedra, and the Proriol-Koorwinder-Dubiner
polynomials on triangles, prisms, and tetrahedra
\cite{dubinerSpectralMethodsTriangles1991,koornwinderTwoVariableAnaloguesClassical1975,proriolFamillePolynomesDeux1957}.
Orthogonality  has also been employed to construct a polynomial
basis with a better stability on polygons used in the virtual element
method (VEM) \cite{berroneOrthogonalPolynomialsBadly2017}.

As mentioned before, the integration in the local mass matrix $M$ on
an interface element needs to be carried out piecewisely according to
the interface geometry so that integration domains are not
necessarily polygonal, and the integrands are not necessarily
polynomials. Hence, from a point of view of practicality and
computational efficiency, it is preferable to adopt a numerical
quadrature to implement the mass matrix $M$.
This means that we will only use the approximate mass matrix $M_q$ in
dealing with interface problems; therefore, our discussion about the
reconstruction will be centered around the approximate mass matrix $M_q$.

The reconstruction starts from a GC-IFE basis $\hlambda_j(\eta,\xi;
C_j^-, C_j^+), ~1 \leq j \leq (m+1)^2$,
in the general form specified by formula
\eqref{eq:basisFun_lambda_j}. Here, we include coefficient
vectors $C_j^-, C_j^+$ in the notation of a basis function
$\hlambda_j(\eta,\xi)$ to emphasize the fact
that it is completely determined by the coefficient vectors $C_j^\pm
= (C_{ij}^\pm)_{i = 1}^{(m+1)^2}$ once the univariate polynomials
$\hp(\xi)$ and $\hq(\eta)$ have been chosen. This basis can be
generated by either the initial construction procedure in
\autoref{sec:method_1} or the one in \autoref{sec:method_general}.
Then, letting $C = (C^-, C^+)$, we will use
\begin{eqnarray*}
\hB(C) = \hB(C^-, C^+)=\big\{\hlambda_j(\eta,\xi; C_j^-, C_j^+), ~1
\leq j \leq (m+1)^2 \big\}
\end{eqnarray*}
to denote the basis of $\hV^m_{\beta}(\hK_F)$. We will use
$M_q(\hB(C))$ instead of $M_q$ to emphasize that
this approximate mass matrix is formed with the basis $\hB(C)$.


First, we note that the coefficient matrices $C^-$ and $C^+$ are
related according to \eqref{eq:matrix_coeff_eq}. Moreover, a
column transformation of $C^-$ and $C^+$ by a nonsingular matrix
$Q$ yields a set of coefficients for another basis of $\hV^m_{\beta}(\hK_F)$, as
stated in the following lemma.
\begin{lemma}\label{lem:change_of_basis}
If $\hB(C)=\hB((C^-,C^+))$ is a basis for $\hV^m_{\beta}(\hK_F)$ and
$Q$ is an invertible $(m+1)^2\times (m+1)^2$ matrix, then $\hB((C^-
Q,C^+ Q))$ is also a basis for $\hV^m_{\beta}(\hK_F)$.
\end{lemma}

\begin{proof}
First, we note that each  function in
$\hB((C^- Q,C^+ Q))$ is a linear combination of those
in $\hB((C^-,C^+))$. Hence, all functions in $\hB((C^- Q,C^+ Q))$
satisfy the interface jump conditions
\eqref{eqns:IFE_conditions}. Consequently, $\hB((C^- Q,C^+ Q))
\subset \hV^m_{\beta}(\hK_F)$. Moreover,
functions in $\hB((C^- Q,C^+ Q))$ are linearly independent because
$Q$ is nonsingular. Since the dimension of
$\hV^m_{\beta}(\hK_F)$ is $(m+1)^2$, these facts together imply that
$\hB((C^- Q,C^+ Q))$ is a basis of
$\hV^m_{\beta}(\hK_F)$.
\end{proof}

Furthermore, by this lemma and \eqref{eq:MassMatrix_matrixform}, we
have the following
formula for the mass matrix associated with the basis $\hB(\tC) =
\hB((\tC^-, \tC^+)) =
\hB((C^-Q, C^+Q))$ transformed by a nonsingular matrix $Q$:
\begin{align}
\begin{split}
M_q(\hB(\tC)) = & \big(\tC^-\big)^T L^- W^- \big(L^-\big)^T \tC^- +
\big(\tC^+\big)^T L^+ W^+ \big(L^+\big)^T \tC^+ \\
=& \big(C^- Q\big)^T L^- W^- \big(L^-\big)^T \big(C^- Q\big) +
\big(C^+ Q\big)^T L^+ W^+ \big(L^+\big)^T \big(C^+ Q\big) = Q^T M_q(\hB(C)) Q.
\end{split} \label{eqn:mass_congruence_2}
\end{align}
Formula \eqref{eqn:mass_congruence_2} has two implications. First, it
implies that, from a computational perspective where numerical
quadrature is admitted, the mass matrix $M_q$ for the transformed
basis $\hB(\tC)$ is congruent to that of the basis $\hB(C)$. More
significantly, with a given basis $\hB(C)$, this congruence suggests
the possibility for us to find another basis $\hB(\tC)$ such that the
mass matrix $M_q(\hB(\tC))$ is better conditioned than the mass
matrix $M_q(\hB(C))$ that is constructed with the original basis $\hB(C)$.
We have considered the following two approaches to take advantage of
this property of the approximate
mass matrix.

\begin{itemize}
\item
\textit{Approach 1}: This is inspired by the work of Berrone and
Borio on the virtual element method
\cite{berroneOrthogonalPolynomialsBadly2017}, where $Q$ is chosen
such that the mass matrix becomes the identity. Specifically, let
$V_{1}\Lambda V_{1}^{T} = M_q(\hB(C))$ be the singular value
decomposition (SVD) of $M_q(\hB(C))$ and let
$Q_1=V_{1}\Lambda^{-1/2}$. Then, by \eqref{eqn:mass_congruence_2}, we have
\begin{eqnarray*}
M_q(\hB(\tC)) = M_q(\hB(C^-Q_1, C^+Q_1)) =
\Lambda^{-1/2}V^T_{1}\left(V_{1}\Lambda V^T_{1}\right)V_{1}
\Lambda^{-1/2}=I_{(m+1)^2}.
\end{eqnarray*}
This approach modifies the GC-IFE basis $\hB(C)$ to produce a
transformed basis $\hB(\tC)$ that is orthonormal with respect to
the $L^2$ inner product and the chosen quadrature rule. As a
result, the corresponding approximate mass matrix $M_q(\hB(\tC))$
achieves an optimal condition number under the idealized
assumption that all involved computations are exact. While such
an assumption does not hold in finite-precision computations, the
procedure nonetheless confirms the possibility to construct a
GC-IFE basis on an interface element $K$ whose associated mass
matrix has the optimal condition number.

\item
\textit{Approach 2}: This approach is based on a factorization of
$M_q(\hB(C))$ as follows. { If
$\NTq=\mathtt{N}_{\mathtt{q}}^{-} + \mathtt{N}_{\mathtt{q}}^{+}$,  let }
\begin{eqnarray*}
\{\x_r, w_r\}_{r=1}^{{\NTq}} =
\{(\x_k^{-},w_k^{-})\}_{k=1}^{\mathtt{N}_{\mathtt{q}}^{-}} \cup
\{(\x_k^{+},w_k^{+})\}_{k=1}^{\mathtt{N}_{\mathtt{q}}^{+}},
\end{eqnarray*}
then, by \eqref{eq:MassMatrix_matrixform}, we can express $M_q(\hB(C))$ as
\begin{align}
\begin{split}
M_q(\hB(C)) =& \Big ( \sum_{r=1}^{{\NTq}} w_r \hlambda_i(\hx_r; C_i^-,
C_i^+)\hlambda_j(\hx_r; C_j^-, C_j^+) \Big)_{i,j=1}^{(m+1)^2}=
V(\hB(C))^T W V(\hB(C)), \\
\text{with} \hspace{0.1in} W
=&\operatorname{diag}\left(w_1,w_2,\dots,w_{{\NTq}}\right),
\end{split} \label{eq:MassMatrix_matrixform_factored}
\end{align}
and $V(\hB(C))$ is the generalized Vandermonde matrix of the
functions in the basis $\hB(C)$ evaluated at the nodes $\x_r, ~1
\leq r \leq {\NTq}$, i.e.,
\begin{eqnarray*}
V_{r, j}(\hB(C)) = \hlambda_j(\hx_r), ~1 \leq r \leq {\NTq}, ~1
\leq j \leq (m+1)^2.
\end{eqnarray*}
We now factorize $M_q(\hB(C))$ by letting ${\tilde V}(\hB(C))=
\sqrt{W} V(\hB(C))$ so that
$M_q(\hB(C)) = V(\hB(C))^T W V(\hB(C)) = {\tilde V}(\hB(C))^T
{\tilde V}(\hB(C))$. Assuming that
${\NTq} \geq (m+1)^2$, we compute an \textit{economy} SVD of
${\tilde V}(\hB(C))$
such that ${\tilde V}(\hB(C)) = U_{2}\Sigma V_{2}^T$, where $U_{2}$
is ${\NTq}\times (m+1)^2$ and $\Sigma, V_{2}$ are $(m+1)^2\times
(m+1)^2$ matrices, respectively. Letting $Q_2 = V_2\Sigma^{-1}$, we
obtain another basis
$\hB(C^-Q_2, C^+Q_2)$. By \eqref{eqn:mass_congruence_2}, the
approximate mass matrix associated with this
new basis will be such that
\begin{eqnarray*}
M_q(\hB(C^-Q_2, C^+Q_2)) &=& Q_2^T M_q(\hB(C)) Q_2 =
\Sigma^{-1}V_{2}^T {\tilde V}(\hB_{C})^T {\tilde V}(\hB_{(C)})
V_{2}\Sigma^{-1} \\
&=& \Sigma^{-1}V^T_{2} V_{2}\Sigma U^T_{2} U_{2}\Sigma V^T_{2}
V_{2}\Sigma^{-1} = I_{(m+1)^2}.
\end{eqnarray*}

\end{itemize}

We note that these two approaches have identical effect: both methods
produce a GC-IFE basis with an optimally conditioned approximate mass
matrix. Moreover, in case that the eigenvalues of
$M_q(\hB(C))$ are distinct, and with the idealized assumption that
all the involved computations are exact, these two approaches are
mathematically  equivalent, i.e., they produce the same GC-IFE
basis, meaning $\hB(C^-Q_1, C^+Q_1) = \hB(C^-Q_2, C^+Q_2)$, with the
understanding that the corresponding $i$-th basis functions in these
two bases are identical up to sign. To see this, we note that an SVD
of the approximate mass matrix $M_q(\hB(C))$ is established. In
Approach 1, an SVD is directly computed from $M_q(\hB(C))$ such that
$V_{1}\Lambda V_{1}^{T} = M_q(\hB(C))$. In Approach 2, even though an
SVD of the generalized Vandermonde matrix $V(\hB(C))$ is computed,
this leads to another SVD of $M_q(\hB(C))$ as follows:
\begin{align*}
M_q(\hB(C)) =& V(\hB(C))^T W V(\hB(C)) = {\tilde V}(\hB(C))^T {\tilde V}(\hB(C))
= V_2\Sigma U_2^T U_2^T \Sigma V_2^T = V_2 \Sigma^2 V_2^T.
\end{align*}
Comparing this SVD with the previous one $V_{1}\Lambda V_{1}^{T} =
M_q(\hB(C))$ and assuming the diagonal entries in $\Lambda$ and
$\Sigma$ are in descending order, we can see that $\Lambda =
\Sigma^2$. Furthermore,
because the eigenvalues of $M_q(\hB(C))$ are distinct, the
corresponding columns (the left singular vectors
of $M_q(\hB(C))$) in $V_1$ and $V_2$ are the same up to sign.
Therefore, matrices
$Q_1 = V_1\Lambda^{-1/2}$ and $Q_2 = V_2\Sigma^{-1}$ differ only by
signs in their corresponding columns, and this leads to the
conclusion that $\hB(C^-Q_1, C^+Q_1) = \hB(C^-Q_2, C^+Q_2)$ with the
understanding that the corresponding $i$-th basis functions are
identical up to sign. However, when implemented on computers, these
two approaches perform
distinctively, especially for higher degree polynomials. This
discrepancy arises from the fact $\Lambda = \Sigma^2$, which implies
the condition number of $M_q(\hB(C))$ is the square of the condition
number of ${\tilde V}(\hB(C))$. As a result, matrix $M_q(\hB(C))$ can
be significantly more ill-conditioned than ${\tilde V}(\hB(C))$.
Hence, with inexact computations on a computer, the accuracy of the
numerically produced singular values (the diagonal entries of
$\Lambda$) and singular vectors (columns of $V_1$ of $M_q(\hB(C))$ is
limited by the machine precision and the condition number of
$M_q(\hB(C))$. Consequently, the transformation matrix $Q_1$ computed
in Approach 1 by a machine may deviate substantially from the exact
$Q_1$. On the other hand, since the condition number of ${
\tilde V}(\hB(C))$ is smaller, its numerically produced SVD is less
susceptible to the round-off errors on a computer. Thus, the
transformation matrix $Q_2$ based on this SVD in Approach 2 is more
reliable than $Q_1$ in Approach 1 in actual computations. This results
in the approximate mass matrix for the GC-IFE basis produced by
Approach 2 having a superior condition number compared to its
counterpart from Approach 1.

We can also numerically observe the similarities and differences
between the two approaches described above by a simple example where
the interface is the unit circle and the interface element is
$\frac{1}{\sqrt{2}}\left(1+\frac{h}{2}[-1,1]^2\right)$ with diameter
$h=\frac{1}{4}$. First, we construct an initial basis $\B(C^-,C^+)$
according to the procedures in either \autoref{sec:method_1}
or \autoref{sec:method_general} with
$(\beta^-,\beta^+)=(1,10^3)$. We construct the mass matrix $M_q^{(0)}$
associated with this basis. Then, we use this initial basis to
reconstruct two bases $\B(C^-Q_1,C^+Q_1)$ and $\B(C^-Q_2,C^-Q_2)$ by
Approach 1 and Approach 2, and we further construct the mass matrices
$M_q^{(1)}$ and $M_q^{(2)}$ associated with these bases,
respectively. The condition number for these mass matrices are listed
in \autoref{table:cond_M_M1_M2} for degrees $m=1,2,\dots,10$. The
data in this table indicates that the two approaches lead to similar
results for lower degrees $1\le m\le 7$ as expected. However, the
data in \autoref{table:cond_M_M1_M2} demonstrate that Approach 1
suffers from numerical instabilities for $m\ge 8$ due to the
large condition number of $M_q^{(0)}$, which results in $M_q^{(1)}\ne
I$. On the other hand, the corresponding data in
\autoref{table:cond_M_M1_M2} clearly indicates that  Approach 2
is more robust. In conclusion, both theoretical
consideration and numerical
corroboration lead us to recommend Approach 2.
\begin{table}[htbp]
\begin{equation}
\begin{array}{
|c||c|c|c|}\hline
m & \operatorname{cond}(M_q^{(0)}) &
\operatorname{cond}(M_q^{(1)}) & \operatorname{cond}(M_q^{(2)}) \\\hline
1  & 2.3066\mathrm{E+}02 & 1.0000              & 1.0000 \\
2  & 6.4918\mathrm{E+}03 & 1.0000              & 1.0000 \\
3  & 1.7175\mathrm{E+}05 & 1.0000              & 1.0000 \\
4  & 9.5642\mathrm{E+}06 & 1.0000              & 1.0000 \\
5  & 7.1767\mathrm{E+}08 & 1.0000              & 1.0000 \\
6  & 1.0426\mathrm{E+}11 & 1.0000              & 1.0000 \\
7  & 4.2865\mathrm{E+}12 & 1.0000              & 1.0000 \\
8  & 3.4601\mathrm{E+}15 & 1.5355              & 1.0000 \\
9  & 1.6926\mathrm{E+}17 & 1.8664\mathrm{E+}01 & 1.0000 \\
10 & 1.8349\mathrm{E+}20 & 8.4148\mathrm{E+}04 & 1.0000 \\\hline
\end{array}
\notag
\end{equation}
\caption{The condition number of the mass matrix before and after
reconstructing the basis.}
\label{table:cond_M_M1_M2}
\end{table}
\section{A numerical example for illustrating key computations}
\label{sec:NumExample}

In this section, we present a simple program that illustrates the steps
described in the previous section. The purpose of this program is to construct
the IFE basis functions on each interface element, and compute the $L^2$
projection of a given function $u$ onto the global DG space. Then, the $L^2$
norm of the projection error is calculated. A similar approach can be
followed to
calculate the solution to an elliptic or hyperbolic interface problem using an
interior penalty formulation
\cite{adjeridHighOrderGeometry2024,meghaichiHigherOrderImmersed2024}. Here, we
only present the steps pertinent to the Frenet IFE spaces and the computations
associated with  them.
For the full
version of the program, we refer the reader to
\cite{adjeridMATLABImplementationGeometryConforming}.

For simplicity, we consider Cartesian meshes having $\mathtt{n}
\times \mathtt{n}$ elements on
$[-1,1]^2$. Assuming that the geometric information of the interface is
specified in \verb!CurveDG!, the following code snippet generates a uniform
mesh, identifies the interface elements by labeling each  element $K$ with $0$
if $K\subset \Omega^-$, $1$ if $K\subset \Omega^+$, and $2$ if $K$ is
an interface
element, and collects the geometric information of each interface element.\\

\noindent
\begin{minipage}{\linewidth}
\begin{matlabblock}
mesh=uniform_rectangle_mesh_generator([-1, 1], [-1, 1], n+1, n+1);
typelist=ElemTypeList(mesh, CurveDG);
IntElemList=find(typelist==2);
xiGuessList=xiInitGuess(mesh, IntElemList, CurveDG);
IntElemInfoList=IntElemInfoLister(mesh, IntElemList, xiGuessList, CurveDG);
\end{matlabblock}
\end{minipage}

Then a basis of each local Frenet IFE space is constructed on each interface
element and stored as two three-dimensional arrays. The first array stores the
matrices $C^-$, and the second stores the matrices $C^+$. As discussed
previously, the construction of the bases can be made efficient by {caching} the
quantities shown in blue in \eqref{eq:matrix_A_entries} and reusing them on each
interface element.\\

\noindent
\begin{minipage}{\linewidth}
\begin{matlabblock}
[Cm_list, Cp_list]=ImmFESpaces(mesh, CurveDG, betas, degree, IntElemInfoList,...
IntElemList, pPolyFun, qPolyFun,"Vandermonde");
\end{matlabblock}
\end{minipage}

The last argument \verb!"Vandermonde"! indicates that Approach 2 was used
to orthogonalize the basis. Alternatively, one
can replace this argument with \verb!"Mass"! to use Approach 1, or
\verb!"None"! to
skip the reconstruction step.
Now, we loop over all elements to compute the $L^2$ projection of a given
function and the $L^2$ norm of the projection error. For conciseness,
we only show the
portion of the code that deals with interface elements. The portion for
non-interface elements is standard and can be found in many textbooks on the DG
method, \textit{e.g.},
\cite{hesthavenNodalDiscontinuousGalerkin2008}, or in the full
program.

\begin{matlabblock}
N_elements=size(mesh.e, 2);
L2_errors=zeros(N_elements, 1);

for k=1:N_elements
    elem=mesh.p(:, mesh.e(:, k));
    if typelist(k)==2
        %Element information
        kk=find(IntElemList==k);
        IntElemInfo=IntElemInfoList(kk);
        xi_mid=IntElemInfo.xi_mid;
        %Quadrature rule
        [xy_m, w_m, xy_p, w_p]=TraditionalQuadrature(elem, CurveDG, xi_mid, degree+1);
        etaxi_m=FrenetRmap(xy_m, CurveDG, xi_mid, 10);
        etaxi_p=FrenetRmap(xy_p, CurveDG, xi_mid, 10);
        %Generalized Vandermonde matrix
        L_m=ImmVandermondeMat(etaxi_m, degree, pPolyFun, qPolyFun, IntElemInfo);
        L_p=ImmVandermondeMat(etaxi_p, degree, pPolyFun, qPolyFun, IntElemInfo);
        %The local mass matrix
        C_m=Cm_list(:, :, kk);
        C_p=Cp_list(:, :, kk);
        M=C_m'*L_m'*(w_m.*L_m*C_m)+C_p'*L_p'*(w_p.*L_p*C_p);
        %The load vector
        u_m=u(xy_m);
        u_p=u(xy_p);
        b_vect=(L_m*C_m)'*(w_m.*u_m(:))+(L_p*C_p)'*(w_p.*u_p(:));
        % The local L^2 projection and the norm of the projection error.
        uh=M\b_vect;
        L2_errors(k)=sqrt(dot((L_m*C_m*uh-u_m(:)).^2, w_m)+dot((L_p*C_p*uh-u_p(:)).^2, w_p));
    else %non interface element
        `\mlplaceholder{Perform a tradtional $L^2$ projection on non-interface elements (see the full script).}`
    end

end
\end{matlabblock}

The same process can be used to numerically test the approximation
capabilities of the local IFE
space discussed in \cite{adjeridHighOrderGeometry2024}. For instance, consider
the domain $\O=[-1,1]^2$ split by the interface $x^2+y^2=r_0^2$ where
$r_0=\frac{1}{\sqrt{3}}$ into
$\O^+=\{(x,y)\in \O \mid x^2+y^2>r_0^2\}$, and $\O^- =\{(x,y)\in \O \mid
x^2+y^2<r_0^2\}$. Consider, the function $u:\O \to \mathbb{R}$
\[
u(x,y)=
\begin{cases}
\frac{1}{\beta^+} \cos(2\pi r^2),& r>r_0,\\
\frac{1}{\beta^-} \cos(2\pi r^2)+ \cos(2\pi r_0^2)\left(
\frac{1}{\beta^+}-\frac{1}{\beta^-}
\right)  ,& r<r_0,
\end{cases}
\qquad r=\sqrt{x^2+y^2},
\]
where $(\beta^-,\beta^+)=(1000,1)$. In the construction, we follow Approach 2
described in \autoref{subsec:Reconstruction} to obtain an orthonormal
basis. After the construction, we calculate the $L^2$ projection of $u$ onto the
local space on each element, and report the global $L^2$ norm of the
projection error
$u-P_hu$ in \autoref{table:projection_error}, where we observe that
the norm of  the projection error decays at an optimal rate under
mesh refinement, as
shown in \cite{adjeridHighOrderGeometry2024}.
\begin{table}[ht]
\[
\begin{array}{|c||c|c||c|c||c|c||c|c|}
\hline
& \multicolumn{2}{|c||}{m=1}
& \multicolumn{2}{|c||}{m=2}
& \multicolumn{2}{|c||}{m=3}
& \multicolumn{2}{|c|}{m=4}
\\ \hline
N&
\norm{u-P_h u}_{L^2(\O)} & \text{rate}&
\norm{u-P_h u}_{L^2(\O)} & \text{rate}&
\norm{u-P_h u}_{L^2(\O)} & \text{rate}&
\norm{u-P_h u}_{L^2(\O)} & \text{rate}\\
\hline
16 & 8.1386\mathrm{E-}02 & \text{---} & 9.2883\mathrm{E-}03 &
\text{---} & 8.8048\mathrm{E-}04 & \text{---} & 7.5479\mathrm{E-}05 &
\text{---} \\
32 & 2.0798\mathrm{E-}02 & 1.9683 & 1.1914\mathrm{E-}03 & 2.9627 &
5.7397\mathrm{E-}05 & 3.9392 & 2.3860\mathrm{E-}06 & 4.9834 \\
64 & 5.2312\mathrm{E-}03 & 1.9912 & 1.5034\mathrm{E-}04 & 2.9864 &
3.6213\mathrm{E-}06 & 3.9864 & 7.4910\mathrm{E-}08 & 4.9933 \\
128 & 1.3098\mathrm{E-}03 & 1.9978 & 1.8843\mathrm{E-}05 & 2.9961 &
2.2688\mathrm{E-}07 & 3.9965 & 2.3437\mathrm{E-}09 & 4.9983 \\
256 & 3.2756\mathrm{E-}04 & 1.9995 & 2.3568\mathrm{E-}06 & 2.9991 &
1.4189\mathrm{E-}08 & 3.9991 & 7.3261\mathrm{E-}11 & 4.9996 \\\hline
\end{array}
\]
\caption{$L^2$ projection errors and convergence rate for $m=1,2,3$ and $4$.}
\label{table:projection_error}
\end{table}

\section{Reproducibility}
The source code for reproducing the results shown in \autoref{fig:cond_A},
\autoref{fig:cond_At}, \autoref{table:cond_M_M1_M2} and
\autoref{table:projection_error} can be found in
\cite{adjeridMATLABImplementationGeometryConforming}.

\section{Conclusion}\label{sec:conclusion} 
In this work, we presented a detailed discussion of the construction
of the local geometry-conforming
immersed function element basis functions, and proposed a new method for
constructing orthonormal bases using the  SVD decomposition of the
generalized Vandermonde matrix. The theoretical procedures in this article are
accompanied by \MATLAB scripts to illustrate the key steps in the
construction and
the implementation of the basis functions, and to facilitate the adoption of the
GC-IFE  method.



\bibliographystyle{plainurl}

\end{document}